\documentclass[12pt,reqno]{amsart}

\usepackage{epsfig}
\usepackage{amsmath, amsthm, amsfonts, amssymb, mathrsfs}
\usepackage{tikz-cd}
\usepackage{graphicx}

\numberwithin{equation}{section}
\DeclareMathOperator\Hom{Hom}
\DeclareMathOperator\Ext{Ext}

\newcommand{\D}{{\mathbb D}}

\newcommand{\R}{{\mathbb R}}
\newcommand{\C}{{\mathbb C}}

\newcommand{\Z}{{\mathbb Z}}

\newcommand{\K}{{\mathbb K}}

\newcommand{\dcal}{{\mathcal D}}

\newtheorem{theo}{{\sc \bf Theorem}}[section]

\newtheorem{lem}[theo]{{\sc \bf Lemma}}
\newtheorem{prop}[theo]{{\sc \bf Proposition}}

\newenvironment{defin}{\medskip\noindent{\bf Definition:\/} }{\medskip}

\begin{document}

\title{Noncommutative Geometry of the Quantum Disk}

\author[Klimek]{Slawomir Klimek}
\address{Department of Mathematical Sciences,
Indiana University-Purdue University Indianapolis,
402 N. Blackford St., Indianapolis, IN 46202, U.S.A.}
\email{sklimek@math.iupui.edu}

\author[McBride]{Matt McBride}
\address{Department of Mathematics and Statistics,
Mississippi State University,
175 President's Cir., Mississippi State, MS 39762, U.S.A.}
\email{mmcbride@math.msstate.edu}

\author[Peoples]{J. Wilson Peoples}
\address{Department of Mathematics,
Pennsylvania State University,
107 McAllister Bld., University Park, State College, PA 16802, U.S.A.}
\email{jwp5828@psu.edu}

\thanks{We would like to thank M. Khalkhali for helpful comments.}

\date{\today}

\begin{abstract}
We discuss various aspects of noncommutative geometry of a smooth subalgebra of the Toeplitz algebra. In particular, we study the structure of derivations on this subalgebra.
\end{abstract}

\maketitle
\section{Introduction}
The purpose of this paper is to study a smooth subalgebra $\mathcal{T}^\infty$ of the Toeplitz C$^*$-algebra, $\mathcal{T}$, an object that captures a smooth structure on the quantum disk \cite{KL}. We also study the related algebra $\mathcal{K}^\infty$ of smooth compact operators with respect to a basis of the Hilbert space. This work is a continuation of our previous papers on the subject of noncommutative geometry of various examples of quantum spaces, in particular \cite{KMRSW1}, \cite{KMRSW2}, \cite{KM6}.

Smooth compact operators appeared previously in Phillips' paper \cite{Ph} as well as in \cite{ENN}. The smooth Toeplitz algebra was introduced by Cuntz \cite{Cu}, see also \cite{Kh} and a recent paper \cite{Pi}. Though a portion of the content of this paper has been described before, this is an attempt at greater detail. Among several noncommutative geometry applications, smooth algebras play an important role in cyclic cohomology theory \cite{Me}. 

The noncommutative disk is a relatively simple example of a noncommutative space and hence it is a good testing ground to illustrate various general concepts of noncommutative geometry \cite{Connes}. The quantum disk has many properties analogous to the classical unit disk $\D$. It is a C$^*$-algebra generated by a single generator  whose norm is 1 and the spectrum is all of $\D$. There is a short exact sequence:
\begin{equation*}
0\rightarrow \mathcal{K}\rightarrow\mathcal{T}\rightarrow C(S^1)\rightarrow 0 ,
\end{equation*}
which should be compared with the short exact sequence for the classical disk:
\begin{equation*}
0\rightarrow C_0(\D)\rightarrow C(\D)\rightarrow C(S^1)\rightarrow 0,
\end{equation*}
where $C_0(\D)$ are the continuous functions on the disk that vanish on the boundary. The quantum disk $\mathcal{T}$ has natural noncommutative polar coordinates \cite{KMRSW1}, employed for example in the equation \eqref{poly_U} below, which is an analog of the partial polar Fourier decomposition of functions on the usual unit disk. Additionally, the group SU(1,1)/$\Z_2$ of holomorphic automorphisms of $\D$ acts in an analogous way on the noncommutative disk $\mathcal{T}$, see \cite{KL}.

The smooth subalgebra of the Toeplitz C$^*$-algebra is made of Toeplitz operators with smooth symbols and of smooth compact operators. We explicitly construct appropriate Frechet structures on $\mathcal{K}^\infty$ and $\mathcal{T}^\infty$ and prove that those algebras are closed under holomorphic calculus so that they have the same K-Theory as their corresponding C$^*$-algebra closures. We also verify that they are closed under smooth functional calculus of self-adjoint elements.

Most of the paper is focused on describing derivations on relevant C$^*$-algebra and on their smooth subalgebras. In particular, using results from \cite{KMRSW1}, \cite{KMRSW2}, we classify derivations on $\mathcal{T}^\infty$ and show that, up to inner derivations, they are lifts of derivations on $C^\infty(\R/\Z)$, the factor algebra of $\mathcal{T}^\infty$ modulo the ideal $\mathcal{K}^\infty$ of smooth compact operators. Additionally, we discuss an action of M\"obius transformations on $\mathcal{T}^\infty$ and various aspect of K-Theory, K-Homology and cyclic cohomology of the quantum disk.

The paper is organized as follows.  In section 2 we introduce and study smooth compact operators. Section 3 contains a detail discussion of smooth Toeplitz algebra, while in sections 4 and 5 we investigate the structure and classifications of derivations and M\"obius automorphisms. Finally, section 6 contains remarks on K-Theory, K-Homology and cyclic cohomology.

\section{Smooth Compact Operators}
\subsection{Basic Definitions}
Let $\mathcal{K}(H)$ be the algebra of compact operators on a (separable) Hilbert space $H$.  A choice of an orthonormal basis $\{E_k\}_{k\ge0}$ of $H$ determines a system of units $\{P_{k,l}\}_{k,l\ge0}$ in $\mathcal{K}(H)$ that satisfy the following relations:
\begin{equation*}
P_{k,l}^* = P_{l,k} \quad\textrm{and}\quad P_{k,l}P_{r,s} = \chi_{l,r}P_{k,s}\,,
\end{equation*}
where $\chi_{l,r}=1$ for $l=r$ and is equal to zero otherwise.  The {\it smooth compact operators} with respect to $\{E_k\}$ is the set of operators of the form
\begin{equation*}
c = \sum_{k,l\ge0}c_{k,l}P_{k,l}\,,
\end{equation*}
so that the coefficients $\{c_{k,l}\}_{k,l\ge0}$ are rapidly decaying (RD).  We denote the set of smooth compact operators by $\mathcal{K}^\infty(H)$ and since $H$ is fixed here, we drop it in the notation for brevity.

This definition does depend on the choice of a basis for $H$.  For example, if $x$ is a unit vector in $H$ which is not a RD linear combination of $\{E_k\}$, then the orthogonal projection  $P_{\{x\}}$ onto the line spanned by $x$ is not a smooth compact operator with respect to the basis $\{E_k\}$.  However, it is smooth with respect to any basis which contains $x$.  

\subsection{0,N - Norms}
We will rephrase the definition of smooth compact operators using norms and put a Fr\'{e}chet space structure on $\mathcal{K}^\infty$. These norms are more complicated than those in \cite{Ph} but offer some technical advantages.

Let $\K$ be the following unbounded, positive diagonal operator:
\begin{equation*}
\K E_k = kE_k\,.
\end{equation*}
It is easy to see the following relations with $P_{k,l}$:
\begin{equation*}
\K P_{k,l} = kP_{k,l}\quad\textrm{and}\quad P_{k,l}\K = lP_{k,l}\,.
\end{equation*}
For any bounded operator $a$ in $H$ we set
\begin{equation*}
\|a\|_{0,N} = \|a(I+\K)^N\|\in[0,\infty]\,,
\end{equation*}
where $N\in\Z_{\ge0}$. 

The proposition below summarizes the basic properties of $\|\cdot\|_{0,N}$.
\begin{prop}\label{N-norm_basics}
Let $a$ and $b$ be bounded operators in $H$, then
\begin{enumerate}
\item $\|a\|_{0,0}=\|a\|$.
\item $\|a\|_{0,N}\le \|a\|_{0,N+1}$.
\item $\|a\|_{0,N}\le \|a(I+\K)^N\|_{\textrm{HS}}\le\sqrt{\pi^2/6}\cdot\|a\|_{0,N+1}$, where $\|\cdot\|_{\textrm{HS}}$ is the Hilbert-Schmidt norm in on $B(H)$.
\item $\|ab\|_{0,N}\le\|a\|_{0,0}\|b\|_{0,N}\le \|a\|_{0,N}\|b\|_{0,N}$.
\end{enumerate}
\end{prop}
Note that if $a$ and $b$ are in $\mathcal{K}^\infty$, then clearly all of the above norms in the proposition are finite.
\begin{proof}
The statement $(1)$ is immediate from the definition $\|\cdot\|_{0,N}$.  Statement $(2)$ follows from the fact that
\begin{equation*}
\|(I+\K)^{-1}\|=1\,.
\end{equation*}
Since $\|a\|\le\|a\|_{\textrm{HS}}$, the first inequality in statement $(3)$ follows.  To see the second inequality in statement $(3)$, notice that
\begin{equation*}
\begin{aligned}
 \|a(I+\K)^N\|_{\textrm{HS}} &= \|a(I+\K)^{N+1}(I+\K)^{-1}\|_{\textrm{HS}}\le \|a(I+\K)^{N+1}\|\,\|(I+\K)^{-1}\|_{\textrm{HS}} \\
&= \|a(I+\K)^{N+1}\|\left(\sum_{k=0}^\infty \frac{1}{(1+k)^2}\right)^{1/2} = \sqrt{\frac{\pi^2}{6}}\cdot\|a(I+\K)^{N+1}\|
\end{aligned}
\end{equation*}
where $\|ab\|_{\textrm{HS}}\le \|a\|\|b\|_{\textrm{HS}}$ was used.  Statement $(1)$ and $(2)$ imply that $\|a\|\le\|a\|_{0,N}$ and thus statement $(4)$ follows using $\|ab\|_{\textrm{HS}}\le \|a\|\|b\|_{\textrm{HS}}$ again.  This completes the proof.
\end{proof}

Notice that statement $(3)$ means that the norms $\|\cdot\|_{0,N}$ are equivalent to the norms $ \|a(I+\K)^N\|_{\textrm{HS}}$. Additionally, if $\{a_{kl}\}$ are the matrix coefficients of $a$ in the basis $\{E_k\}$, then
\begin{equation*}
\begin{aligned}
 \|a(I+\K)^N\|_{\textrm{HS}}^2 &= \textrm{tr}\left((I+\K)^Na^*a(I+\K)^N\right) = \textrm{tr}\left((I+\K)^{2N}\sum_{k,l}\overline{a}_{k,l}P_{l,k}\sum_{r,s}a_{r,s}P_{r,s}\right) \\
&=\textrm{tr}\left(\sum_{k,l,s}(1+l)^{2N}\overline{a}_{k,l}a_{k,s}P_{l,s}\right) = \sum_{k,l}(1+l)^{2N}|a_{k,l}|^2\,.
\end{aligned}
\end{equation*}
Thus, the existence of $\|a\|_{0,N}$ is equivalent to the convergence of the positive series above.  In fact, the above calculation shows that $\|a\|_{0,N}$ captures the RD property in one of the indices of $a_{k,l}$.

To proceed further we need the following derivation on $\mathcal{K}^\infty$:
\begin{equation*}
\delta_\K(a) = [\K,a]\,.
\end{equation*}
Clearly $\delta_\K$ is linear and satisfies the Leibniz rule as it is a commutator.  Notice that
\begin{equation}\label{der_k_on_proj}
\delta_\K(P_{k,l}) = (k-l)P_{k,l}\quad\textrm{and}\quad \delta_\K(c) = \sum_{k,l}(k-l)c_{k,l}P_{k,l}
\end{equation}
for a given $a\in\mathcal{K}^\infty$.  Since $\{a_{k,l}\}$ is a RD sequence, it follows that $\{(k-l)a_{k,l}\}$ is a RD sequence and hence $\delta_\K:\mathcal{K}^\infty\to\mathcal{K}^\infty$ is well-defined.  Moreover, from the definition of $\delta_\K$ we have $\delta_\K(a^*)=-(\delta_\K(a))^*$.  The next proposition describes an important property $\|\cdot\|_{0,N}$ enjoys in relation to the adjoint.

\begin{prop}\label{N-norm_star}
Given a bounded operator $a$, we have
\begin{equation*}
\|a^*\|_{0,N} \le \sum_{j=0}^N\begin{pmatrix} N \\ j \end{pmatrix}\|\delta_\K^j(a)\|_{0,N}\,.
\end{equation*}
\end{prop}

\begin{proof}
For $N=0$, the inequality is trivial as $\|a^*\| = \|a\|$.  We proceed by induction.  Suppose the inequality is true for some $N$.  Notice that $\|a(I+\K)\|_{0,N} = \|a\|_{0,N+1}$.  Using this fact, Proposition \ref{N-norm_basics} and the induction hypothesis we have:
\begin{equation*}
\begin{aligned}
&\|a^*\|_{0,N+1} = \|a^*(I+\K)\|_{0,N} =\|((I+\K)a)^*\|_{0,N}\le \sum_{j=0}^N\begin{pmatrix} N \\ j \end{pmatrix}\|\delta_\K^j((I+\K)a)\|_{0,N} \\
&%=\sum_{j=0}^N \begin{pmatrix} N \\ j \end{pmatrix}\|\delta_\K^j(a(I+\K)) + \delta_\K^{j+1}(a)\|_{0,N} 
\le \sum_{j=0}^N \begin{pmatrix} N \\ j \end{pmatrix} \|\delta_\K^j(a(I+\K))\|_{0,N} + \sum_{j=0}^N \begin{pmatrix} N \\ j \end{pmatrix}\|\delta_\K^{j+1}(a)\|_{0,N} \\
%&=\sum_{j=0}^N \begin{pmatrix} N \\ j \end{pmatrix}\|\delta_\K^j(a)(I+\K)\|_{0,N} + \sum_{j=1}^{N+1} \begin{pmatrix} N \\ j-1 \end{pmatrix}\|\delta_\K^j(a)\|_{0,N} \\
&\le \sum_{j=0}^N \begin{pmatrix} N \\ j \end{pmatrix}\|\delta_\K^j(a)\|_{0,N+1} + \sum_{j=1}^{N+1} \begin{pmatrix} N \\ j-1 \end{pmatrix}\|\delta_\K^j(a)\|_{0,N+1}= \sum_{j=0}^{N+1} \begin{pmatrix} N+1 \\ j \end{pmatrix}\|\delta_\K^j(a)\|_{0,N+1}\,.
\end{aligned}
\end{equation*}
\end{proof}

\subsection{M,N - Norms}
As mentioned earlier, $\|\cdot\|_{0,N}$ only captures the RD property in one index, so, to encompass both indices, we define the convenient norms based on $\|\cdot\|_{0,N}$.  For $M$ and $N$ in $\Z_{\ge0}$, given a bounded operator $a$ in $H$ we define:
\begin{equation*}
\|a\|_{M,N} = \sum_{j=0}^M\begin{pmatrix} M \\ j \end{pmatrix}\|\delta_\K^j(a)\|_{0,N}\,,
\end{equation*}
with $\delta_\K^0(a):=a$. 
Some care has to be taken to interpret the above expression, since $\delta_\K$ is an unbounded derivation. One way to handle this is to notice that $\|\cdot\|_{M,N}$ norms are equivalent, by Proposition \ref{N-norm_star}, to seminorms $\|\delta_\K^j(a)(I+\K)^N\|_{\textrm{HS}}$. In terms of the matrix coefficients we have
\begin{equation}\label{delta_matrix}
\|\delta_\K^j(a)(I+\K)^N\|_{\textrm{HS}}^2 = \sum_{k,l}(1+l)^{2N}(k-l)^{2j}|a_{k,l}|^2\,.
\end{equation}
Consequently, we can make sense of $\|a\|_{M,N}$ as a number in $[0,\infty]$ for any bounded operator.   The following proposition summarizes the basic properties of $\|\cdot\|_{M,N}$.

\begin{prop}\label{MN-norm_basics}
Let $a$ and $b$ be bounded operators in $H$, then
\begin{enumerate}
\item $a\in\mathcal{K}^\infty$ if and only if $\|a\|_{M,N}<\infty$ for all nonnegative integers $M$ and $N$.
\item $\|a\|_{M+1,N} = \|a\|_{M,N} + \|\delta_\K(a)\|_{M,N}$.
\item $\|a\|_{M,N} \le \|a\|_{M,N+1}$.
\item $\|ab\|_{M,N} \le \|a\|_{M,0}\|b\|_{M,N} \le \|a\|_{M,N}\|b\|_{M,N}$.
\item $\|\delta_\K(a)\|_{M,N}\le \|a\|_{M+1,N}$.
\item $\|a^*\|_{M,N}\le \|a\|_{M+N,N}$.
\item $\mathcal{K}^\infty$ is a complete topological vector space.
\end{enumerate}
\end{prop}

\begin{proof}
Statement (1) follows from the equivalence of $\|\cdot\|_{M,N}$ with the following seminorms $\|\delta_\K^j(a)(I+\K)^N\|_{\textrm{HS}}$ along with the formula  \eqref{delta_matrix} for the later norms in terms of the matrix coefficients.  Using the binomial coefficient theorem and resumming we have
\begin{equation*}
\begin{aligned}
\|a\|_{M+1,N} &= \sum_{j=0}^{M+1} \begin{pmatrix} M+1 \\ j \end{pmatrix}\|\delta_\K^j(a)\|_{0,N} = \sum_{j=0}^{M+1}\left(\begin{pmatrix} M \\ j \end{pmatrix} + \begin{pmatrix} M \\ j-1 \end{pmatrix}\right)\|\delta_\K^j(a)\|_{0,N} \\
&=\sum_{j=0}^M \begin{pmatrix} M \\ j \end{pmatrix} \|\delta_\K^j(a)\|_{0,N} + \sum_{j=0}^M \begin{pmatrix} M \\ j \end{pmatrix}\|\delta_\K^j(\delta_\K(a))\|_{0,N} = \|a\|_{M,N} + \|\delta_\K(a)\|_{M,N}\,.
\end{aligned}
\end{equation*}
This proves statement (2).  Statement (3) is a consequence of Proposition \ref{N-norm_basics}.  Proving statement (4) requires induction on $M$, the submultiplicative property in $N$ from Proposition \ref{N-norm_basics} and statement (2).  Indeed, the case $M=0$ is just the submultiplicative property of norms $\|\cdot\|_{0,N}$.  Now suppose the statement is true for some $M$.  Using statement (2) and the induction hypothesis, we get
\begin{equation*}
\begin{aligned}
&\|ab\|_{M+1,N} = \|ab\|_{M,N} + \|\delta_\K(ab)\|_{M,N} \le \|a\|_{M,0}\|b\|_{M,N} + \|\delta_\K(a)b\|_{M,N} + \|a\delta_\K(b)\|_{M,N} \\
%&\le \|a\|_{M,N}\|b\|_{M,N} + \|\delta_\K(a)\|_{M,N}\|b\|_{M,N} + \|a\|_{M,N}\|\delta_\K(b)\|_{M,N} \\
&\le \|a\|_{M,0}\|b\|_{M,N} + \|\delta_\K(a)\|_{M,0}\|b\|_{M,N}+ \|a\|_{M,0}\|\delta_\K(b)\|_{M,N} + \|\delta_\K(a)\|_{M,0}\|\delta_\K(b)\|_{M,N} \\
&=\|a\|_{M+1,0}\|b\|_{M+1,N}\,.
\end{aligned}
\end{equation*}
Statement (5) immediately follows from (2).
%\begin{equation*}
%\begin{aligned}
%\|\delta_\K(a)\|_{M,N} &= \sum_{j=0}^M \begin{pmatrix} M \\ j \end{pmatrix} \|\delta_\K^{j+1}(a)\|_{0,N} = \sum_{j=1}^{M+1} \begin{pmatrix} M \\ j-1 \end{pmatrix} \|\delta_\K^j(a)\|_{0,N} \le \sum_{j=1}^{M+1} \begin{pmatrix} M+1 \\ j \end{pmatrix} \|\delta_\K^j(a)\|_{0,N} \\
%&\le \sum_{j=0}^{M+1} \begin{pmatrix} M+1 \\ j \end{pmatrix} \|\delta_\K^j(a)\|_{0,N} = \|a\|_{M+1,N}\,.
%\end{aligned}
%\end{equation*}
To prove statement (6), we proceed by induction on $M$.  The case $M=0$ reduces to Proposition \ref{N-norm_star}.  Now suppose the statement is true for some $M$.  By using statement (3) and the induction hypothesis, we have
\begin{equation*}
\begin{aligned}
\|a^*\|_{M+1,N} &=\|a^*\|_{M,N} + \|\delta_\K(a^*)\|_{M,N} = \|a^*\|_{M,N} + \|(\delta_\K(a))^*\|_{M,N} \\
&\le \|a\|_{M+N,N} + \|\delta_\K(a)\|_{M+N,N} = \|a\|_{M+N+1,N}\,.
\end{aligned}
\end{equation*}
Statement (7) follows from the usual completeness arguments for spaces of sequences.
\end{proof}
This proposition implies that $\mathcal{K}^\infty$ is a Fr\'{e}chet $*$-algebra with respect to the norms $\|\cdot\|_{M,N}$.

\subsection{Estimating Norms}
There are many other choices of norms or seminorms giving the same topology on $\mathcal{K}^\infty$. Some of the advantages of norms $\|\cdot\|_{M,N}$ is that they are submultiplicative and they can be worked with without using matrix coefficients of operators. In practice, to establish that norms $\|a\|_{M,N}$ are finite, it is obviously enough to establish that the seminorms  
$$\|\delta_\K^l(a)\|_{0,N}$$ are finite for every $l$ and $N$. Other equivalent choices are above considered Hilbert-Schmidt based seminorms 
$$\|\delta_\K^j(a)(I+\K)^N\|_{\textrm{HS}}$$ and also 
$$\|(I+\K)^Na(I+\K)^M\|.$$

Another convenient option is given by the seminorms below.
Consider the following continuous derivations on $\mathcal{K}^\infty$:
\begin{equation*}
\partial_l(a) = [(I+\K)^l,a]\,.
\end{equation*}
Clearly, we have $\partial_1=\delta_{\K}$. Additionally, the following relations, easily established by induction, relate $\partial_l$ with powers of $\delta_{\K}$:
\begin{equation}\label{partial_delta}
\partial_l(a) = \sum_{j=1}^{l}\begin{pmatrix} l \\ j \end{pmatrix} \delta_{\K}^j(a)(I+\K)^{l-j}\,.
\end{equation}
Similarly, we have:
\begin{equation*}
\delta_{\K}^l(a) = \sum_{j=1}^{l} (-1)^{l-j}\begin{pmatrix} l \\ j \end{pmatrix} \partial_j(a)(I+\K)^{l-j}\,.
\end{equation*}
Here we use the convention $(I+\K)^0=I$. Those formulas lead to the following observation.
\begin{prop}\label{equiv_norms}
The norms $\|a\|_{M,N}$, $M,N=0,1,2\ldots$ are equivalent to the seminorms $\|\partial_l(a)\|_{0,N}$, $l,N=0,1,2\ldots$.
\end{prop}

\begin{proof} This is a direct consequence of the definition of $\|\cdot\|_{M,N}$ norms and the relations between $\partial_l$ and $\delta_{\K}$ above.
\end{proof}

\subsection{Fourier Series for Smooth Compact Operators}
Let $U:H\to H$ be the following shift operator:
\begin{equation*}
UE_k = E_{k+1}\,.
\end{equation*}
This operator is an isometry, $U^*U=I$ and it could be used for the following useful, alternative expansion of elements in $\mathcal{K}^\infty$.  If $f:\Z_{\ge0}\to\C$ is a sequence of numbers, we write $f(\K)$ for the diagonal operator
\begin{equation*}
f(\K)E_k = f(k)E_k
\end{equation*}
for $k\in\Z_{\ge0}$.  The diagonal operator $f(\K)$ and the shift $U$ satisfy the following commutation relation:
\begin{equation}\label{comm_rel}
f(\K)U = Uf(\K + I)\,.
\end{equation}
Consider operators given by the infinite, formal series
\begin{equation}\label{series_rep}
a= \sum_{n\ge0} U^na_n(\K) + \sum_{n<0}a_n(\K)(U^*)^{-n}\,.
\end{equation}
Notice that we have:
\begin{equation}\label{delta_k_1}
\delta_\K(U)=U,\ \delta_\K(U^*)=-U^*,\textrm{ and } \delta_\K(f(\K))=0.
\end{equation}  
Thus, if $a$ is of the form given by equation \eqref{series_rep}, we have
\begin{equation*}
\delta_\K(a) = \sum_{n\ge0}nU^n a_n(\K) + \sum_{n<0}na_n(\K)(U^*)^{-n}\,.
\end{equation*}
Moreover, a direct computation shows that
\begin{equation*}
\|a(I+\K)^N\|_{\textrm{HS}}^2 = \sum_{n\ge0}\sum_{k\ge0}(1+k)^{2N}|a_n(k)|^2 + \sum_{n<0}\sum_{k\ge0}(1+k-n)^{2N}|a_n(k)|^2
\end{equation*}
and consequently, we have:
\begin{equation*}
\|\delta_\K^j(a)(I+\K)^N\|_{\textrm{HS}}^2 = \sum_{n\ge0}\sum_{k\ge0}n^{2j}(1+k)^{2N}|a_n(k)|^2 + \sum_{n<0}\sum_{k\ge0}n^{2j}(1+k-n)^{2N}|a_n(k)|^2\,.
\end{equation*}
From these formulas we deduce the following proposition.

\begin{prop}
An operator $a$ of the form given by equation \eqref{series_rep} is in $\mathcal{K}^\infty$ if and only if $\{a_n(k)\}_{n\in\Z,k\in\Z_{\ge0}}$ is a RD sequence in both $n$ and $k$.  Moreover, any $a\in\mathcal{K}^\infty$ can be uniquely written as convergent series given by equation \eqref{series_rep}.
\end{prop}

\begin{proof}
The first part follows from the above formulas for $\|\delta_\K^j(a)(I+\K)^N\|_{\textrm{HS}}$.  
The second part follows from the following consideration. Let $a\in\mathcal{K}^\infty$ and consider the two convergent decompositions:

\begin{equation*}
a =\sum_{k,l\ge0} \tilde{a}_{kl}P_{k,l}\quad\textrm{and}\quad a =\sum_{n\ge0}U^na_n(\K) + \sum_{n<0}a_n(\K)(U^*)^{-n}\,.
\end{equation*}
We compute the relation between $\tilde{a}_{kl}$ and $a_n(k)$.  We have
\begin{equation*}
\begin{aligned}
a&= \sum_{n\ge0}U^na_n(\K) + \sum_{n<0}a_n(\K)(U^*)^{-n} = \sum_{n,k\ge0}U^na_n(k)P_{kk} + \sum_{n<0,k\ge0}a_n(k)P_{kk}(U^*)^{-n}\\
&=\sum_{n,k\ge0}a_n(k)P_{k+n,k} + \sum_{n<0,k\ge0}a_n(k)P_{k,k-n} = \sum_{k\ge l\ge0}a_{k-l}(l)P_{k,l} + \sum_{l>k\ge0}a_{k-l}(k)P_{k,l}\,,
\end{aligned}
\end{equation*}
by resummation.  Thus the relation between the coefficients is:
\begin{equation*}
\tilde{a}_{kl} = \left\{
\begin{aligned}
&a_{k-l}(l) &&\textrm{if }k\ge l\\
&a_{k-l}(k) &&\textrm{if }k<l\,.
\end{aligned}\right.
\end{equation*}
It follows that RD of $a_k(l)$ is equivalent to RD of $\tilde{a}_{kl}$.
\end{proof}

\subsection{Basis Independence}
As before, let $H$ be a separable Hilbert space and let  $\mathcal{E}=\{E_k\}$ be an orthonormal basis of $H$. In this subsection only we write $\mathcal{K}^\infty_\mathcal{E}$ for the space of smooth compact operators with respect to the basis $\mathcal{E}$. We will also use a subscript notation for the corresponding label operators:
\begin{equation*}
\K_\mathcal{E}E_k=kE_k.
\end{equation*}
We have the following version of Proposition \ref{MN-norm_basics}, statement (1):
$a\in\mathcal{K}^\infty_\mathcal{E}$ if and only if 
$$\|(I+\K_\mathcal{E})^Na(I+\K_\mathcal{E})^M\|<\infty$$ 
for all nonnegative integers $M$ and $N$.

Let $S_\mathcal{E}\subset H$ be the dense subspace of $H$ of linear combinations of the basis elements of $\mathcal{E}$ with rapid decay coefficients:
\begin{equation}\label{S_def}
S_\mathcal{E}=\left\{x\in H: x=\sum_{k=0}^\infty x_kE_k, \{x_k\}_{k=0}^\infty \textrm{ is RD}\right\}.
\end{equation}
Then $S_\mathcal{E}$ is a Fr\'{e}chet space with respect to the norms:
$$||x||^2_N:=||(I+\K)^Nx||^2=\sum_{k=0}^\infty (1+k)^{2N}|x_k|^2.$$

Let $\mathcal{F}=\{F_k\}$, $F_k\in S_\mathcal{E}$ be another orthonormal basis of $H$ consisting of elements of $S_\mathcal{E}$. Let $\mathcal{U}$ be the corresponding unitary operator:
\begin{equation*}
\mathcal{U}E_k=F_k.
\end{equation*}
It is easy to see that we have the following relation between label operators:
\begin{equation*}
\K_\mathcal{E}=\mathcal{U}^{-1}\K_\mathcal{F}\mathcal{U}.
\end{equation*}

\begin{prop}
With the above notation, if $\mathcal{U}: S_\mathcal{E}\to S_\mathcal{E}$ is a continuous map of Fr\'{e}chet spaces then 
\begin{equation*}
\mathcal{K}^\infty_\mathcal{F}\subseteq\mathcal{K}^\infty_\mathcal{E}.
\end{equation*}
\end{prop}
\begin{proof}
Continuity of $\mathcal{U}: S_\mathcal{E}\to S_\mathcal{E}$ means that for every $N$ there is $N'$ such that:
\begin{equation*}
||(I+\K)^N\mathcal{U}x||\leq\textrm{const}||(I+\K)^{N'}x||,
\end{equation*}
which is the same as saying that the operator
\begin{equation}\label{S_cont}
(I+\K)^N\mathcal{U}(I+\K)^{-N'}
\end{equation}
is bounded.

Suppose that $a$ is in $\mathcal{K}^\infty_\mathcal{F}$. To show that $a\in\mathcal{K}^\infty_\mathcal{E}$ we want to estimate norms:
\begin{equation*}
\|(I+\K_\mathcal{E})^Na(I+\K_\mathcal{E})^M\|=\|(I+\K_\mathcal{F})^N\mathcal{U}a\,\mathcal{U}^{-1}(I+\K_\mathcal{F})^M\|.
\end{equation*}
Picking appropriate $M'$ and $N'$ so that the operators \eqref{S_cont} are bounded, we see that the above norm expression is equal to:
\begin{equation*}
\left\|\left((I+\K_\mathcal{F})^N\mathcal{U}(I+\K_\mathcal{F})^{-N'}\right)(I+\K_\mathcal{F})^{N'}a\,(I+\K_\mathcal{F})^{M'}\left((I+\K_\mathcal{F})^M\mathcal{U}(I+\K_\mathcal{F})^{-M'}\right)^*\right\|.
\end{equation*}
By the assumptions, the expression is bounded, finishing the proof.
\end{proof}

A completely analogous proof gives the following related result that we will use later.
\begin{prop}\label{Comp_Inv}
With the above notation, if $\mathcal{U}: S_\mathcal{E}\to S_\mathcal{E}$ is a continuous map of Fr\'{e}chet spaces then 
\begin{equation*}
\mathcal{U}\mathcal{K}^\infty_\mathcal{E}\mathcal{U}^{-1}\subseteq\mathcal{K}^\infty_\mathcal{E}.
\end{equation*}
\end{prop}
\begin{proof}
If $a$ is in $\mathcal{K}^\infty_\mathcal{E}$, we want to show that $\mathcal{U}a\,\mathcal{U}^{-1}\in\mathcal{K}^\infty_\mathcal{E}$ by estimating norms:
\begin{equation*}
\|(I+\K_\mathcal{E})^N\mathcal{U}a\,\mathcal{U}^{-1}(I+\K_\mathcal{E})^M\|.
\end{equation*}
Proceeding as above we write this as:
\begin{equation*}
\left\|\left((I+\K_\mathcal{E})^N\mathcal{U}(I+\K_\mathcal{E})^{-N'}\right)(I+\K_\mathcal{E})^{N'}a\,(I+\K_\mathcal{E})^{M'}\left((I+\K_\mathcal{E})^M\mathcal{U}(I+\K_\mathcal{E})^{-M'}\right)^*\right\|,
\end{equation*}
which, for appropriate $M'$ and $N'$, is finite.
\end{proof}

\subsection{Smooth Integral Kernels}
As a concrete example, we consider the Hilbert space $L^2(\R/\Z)$ with its standard basis $\{e^{2\pi ikx}\}_{k\in\Z}$. It turns out that the smooth compact operators $\mathcal{K}^\infty$ in $L^2(\R/\Z)$ with respect to the above basis are precisely the integral operators with smooth kernel.
\begin{prop}
With the above notation, $a\in \mathcal{K}^\infty$ if and only if 
\begin{equation*}
af(x)=\int_0^1a(x,y)f(y)\,dy,
\end{equation*}
where the integral kernel $a(x,y)$ is smooth.
\end{prop}
\begin{proof}
The result follows from a simple calculation showing that the matrix coefficients $a_{k,l}$ of $a$ with respect to the basis $\{e^{2\pi ikx}\}_{k\in\Z}$ and the Fourier series coefficients $\hat a_{k,l}$ of the integral kernel $a(x,y)$ are related by:
\begin{equation*}
a_{k,l}=\hat a_{k,-l}.
\end{equation*}
Consequently, RD of $a_{k,l}$ is equivalent to RD of $\hat a_{k,l}$, the later meaning that $a(x,y)$ is smooth. 
\end{proof}

\subsection{Stability for Smooth Compact Operators}
The last topics we consider in this section is the stability of $\mathcal{K}^\infty$ under the holomorphic calculus and under smooth calculus of self-adjoint elements.

\begin{prop}\label{invert_K_infty}
Suppose that $c\in\mathcal{K}^\infty$ and $I+c$ is invertible in $H$.  Then $(I+c)^{-1}-I$ is in $\mathcal{K}^\infty$.
\end{prop}

\begin{proof}
Notice that
\begin{equation*}
I-(I-c)^{-1} = (I+c)^{-1}c\,.
\end{equation*}
Therefore, it follows from Proposition \ref{N-norm_basics} that
\begin{equation*}
\|I-(I+c)^{-1}\|_{0,N} = \|(I+c)^{-1}c\|_{0,N} \le \|(I+c)^{-1}\|\|c\|_{0,N}<\infty\,.
\end{equation*}
Next, consider the following calculation:
\begin{equation*}
\begin{aligned}
\delta_\K((I+c)^{-1}c) &= \delta_\K((I+c)^{-1})c + (I+c)^{-1}\delta_\K(c) \\
&= -(I+c)^{-1}\delta_\K(c)(I+c)^{-1}c + (I+c)^{-1}\delta_\K(c)\,.
\end{aligned}
\end{equation*}
Notice that both terms in the above equation are of the form: bounded operator times an element from $\mathcal{K}^\infty$.  In fact, by induction we have for any $j$
\begin{equation*}
\delta_\K^j((I+c)^{-1}c) = \sum_i a_ib_i\,,
\end{equation*}
where $a_i$ are bounded operators and $b_i\in\mathcal{K}^\infty$.  It therefore follows that
\begin{equation*}
\|\delta_\K^j((I+c)^{-1}c)\|_{0,N}<\infty
\end{equation*}
for every $N$.  Consequently, for every $M$ and $N$ we get
\begin{equation*}
\|I-(I+c)^{-1}\|_{M,N}<\infty\,,
\end{equation*}
which precisely means that $I-(I+c)^{-1}\in\mathcal{K}^\infty$.
\end{proof}

\begin{theo}
$\mathcal{K}^\infty$ is closed under the holomorphic functional calculus: for any $c\in \mathcal{K}^\infty$ and a function $f$ that is holomorphic on an open domain containing the spectrum of $c$ and such that $f(0)=0$, we have $f(c)\in \mathcal{K}^\infty$.
\end{theo}

\begin{proof}
The usual arguments \cite{Bo} work here as in \cite{KMP2}. Let $c\in \mathcal{K}^\infty$ and let $f$ be a holomorphic function on some open set containing $\sigma(c)$ and $C$ a contour around the spectrum. We define $f(c)$ as
\begin{equation*}
f(c) = \frac{1}{2\pi i}\int_C f(\zeta)(\zeta- c)^{-1}\,d\zeta\,.
\end{equation*}
The convergence of the integral is guaranteed by the completeness of $\mathcal{K}^\infty$, while the condition $f(0)=0$ assures that the outcome is compact
\end{proof}

A key to establish stability under smooth functional calculus is to control exponentials of smooth compact operators.
\begin{prop}
If $c\in\mathcal{K}^\infty$, then $e^c-I$ is in $\mathcal{K}^\infty$.
\end{prop}

\begin{proof} This follows from the expansion into series of smooth compact operators
\begin{equation*}
e^c-I=\sum_{n=1}^\infty\frac{c^n}{n!}
\end{equation*}
and submultiplicativity of $\|\cdot\|_{M,N}$ norms.
\end{proof}
We need a more detailed control of the norms of exponential of self-adjoint compact operators.
\begin{prop}
Suppose that $c\in\mathcal{K}^\infty$ is a self-adjoint smooth compact operator. Then we have estimates
\begin{equation*}
\|e^{ic}-I\|_{0,N}\leq \|c\|_{0,N}\,,
\end{equation*}
and, for $j\geq1$,
\begin{equation*}
\|\partial_j(e^{ic}-I)\|_{0,N}\leq \|\partial_j(c)\|_{0,N}+  \|\partial_j(c)\| \|c\|_{0,N}\,.
\end{equation*}
\end{prop}

\begin{proof} 
First we estimate the $0,N$-norm of $e^{ic}-I$. To do this we write:
\begin{equation*}
e^{ic}-I= \int_0^1\frac{d}{dt}\left(e^{itc}\right)\,dt=i \int_0^1e^{itc}c\,dt\,.
\end{equation*}
It follows from Proposition \ref{N-norm_basics} that we have:
\begin{equation}\label{eic_estimate}
\|e^{ic}-I\|_{0,N}\leq \int_0^1\|e^{itc}\|\|c\|_{0,N}\,dt =\|c\|_{0,N}\,.
\end{equation}

For $j\geq 1$ we write:
\begin{equation}\label{duhamel}
\begin{aligned}
\partial_j(e^{ic})&=(I+\K)^je^{ic}-e^{ic}(I+\K)^j=\int_0^1\frac{d}{dt}\left(e^{i(1-t)c}(I+\K)^je^{itc}\right)\,dt=\\
&=
i\int_0^1e^{i(1-t)c}\partial_j(c)e^{itc}\,dt=i\int_0^1e^{i(1-t)c}\partial_j(c)\,dt+i\int_0^1e^{i(1-t)c}\partial_j(c)(e^{itc}-I)\,dt\,.
\end{aligned}
\end{equation}
Consequently, we estimate, using \eqref{eic_estimate}:
\begin{equation*}
\begin{aligned}
\|\partial_j(e^{ic}-I)\|_{0,N}&\leq \int_0^1\|e^{i(1-t)c}\partial_j(c)\|_{0,N}\,dt+\int_0^1\|e^{i(1-t)c}\partial_j(c)(e^{itc}-I)\|_{0,N}\,dt\leq\\
&\leq \|\partial_j(c)\|_{0,N}+  \|\partial_j(c)\| \|c\|_{0,N}\,,
\end{aligned}
\end{equation*}
by Proposition \ref{N-norm_basics}.
\end{proof}

\begin{theo}
$\mathcal{K}^\infty$ is closed under the holomorphic functional calculus of self-adjoint elements: for any self-adjoint smooth compact operator $c\in\mathcal{K}^\infty$  and any smooth function  $f(x)$ on a neighborhood of the spectrum of $c$ such that $f(0)=0$, we have $f(c)$ is in $\mathcal{K}^\infty$.
\end{theo}

\begin{proof} The proof below is inspired by \cite{BC}. Pick a number $L$ bigger than $2\|c\|$ and extend $f(x)$ to the whole of $\R$ in such a way that it is still smooth and additionally it is $L$-periodic: $f(x+L)=f(x)$ for every $x$. Then $f(x)$ admits a Fourier series representation:
\begin{equation*}
f(x)=\sum_{n\in\Z}f_ne^{2\pi inx/L}
\end{equation*}
with rapid decay coefficients $\{f_n\}$. Taking into account the assumption $f(0)=0$, we have:
\begin{equation*}
f(c)=\sum_{n\in\Z}f_n\left(e^{2\pi inc/L}-I\right).
\end{equation*}
Using the previous proposition we see that the seminorms
\begin{equation}
\|\partial_j(e^{2\pi inc}-I)\|_{0,N}
\end{equation}
grow at most quadratically in $n$, so that $\|\partial_j(f(c))\|_{M,N}<\infty$, by RD of $\{f_n\}$, implying that $f(c)$ is in $\mathcal{K}^\infty$.
\end{proof}

\section{Smooth Toeplitz algebra}
\subsection{Fourier Series for Toeplitz Algebra}
As before, $H$ is a separable Hilbert space with orthonormal basis $\{E_k\}_{k\ge0}$ and $U$ is the shift operator.  The Toepltiz algebra $\mathcal{T}$ is the C$^*$-algebra generated by $U$ and $U^*$, $\mathcal{T} = C^*(U)$.  It was proved in \cite{KMRSW1} that $a$ is a polynomial in $U$ and $U^*$ if and only if $a$ is of the form:
\begin{equation}\label{poly_U}
a= \sum_{n\ge0} U^na_n(\K) + \sum_{n<0}a_n(\K)(U^*)^{-n} \qquad\textrm{finite sums}
\end{equation}
and for every $n$, the functions $k\mapsto a_n(k)$ are eventually constant.

Such a Fourier series can be produced for any $a\in\mathcal{T}$ in the following way. It is easy to see that the map $\rho_\theta:\mathcal{T}\to\mathcal{T}$ given by
\begin{equation}\label{rho_def}
\rho_\theta(a) = e^{2\pi i\K}ae^{-2\pi i\K}\quad\textrm{for}\quad \theta\in\R/\Z
\end{equation}
is a continuous one-parameter group of automorphisms.  This can be verified directly and it also follows from the universality of $\mathcal{T}$ as the universal C$^*$-algebra generated by $U$ and $U^*$ so that $U^*U=I$, known as Coburn's Theorem \cite{Cob}.  Moreover, we have that
\begin{equation*}
\rho_\theta(U) = e^{2\pi i\theta}U\,,\quad \rho_\theta(U^*)=e^{-2\pi i\theta}U^*\,,\quad\textrm{and}\quad\rho_\theta(a(\K)) = a(\K)\,.
\end{equation*}
Thus for a polynomial $a$ in $\mathcal{T}$, given by equation \eqref{poly_U}, we have
\begin{equation*}
\rho_\theta(a) = \sum_{n\ge0}e^{2\pi in\theta}U^na_n(\K) + \sum_{n<0}e^{-2\pi in\theta}a_n(\K)(U^*)^n\,.
\end{equation*}
The fixed point algebra $\mathcal{T}_{\textrm{diag}}$ in $\mathcal{T}$ with respect to $\rho_\theta$ is the algebra of diagonal operators:
\begin{equation*}
\mathcal{T}_{\textrm{diag}} = \{ f(\K) : \lim_{k\to\infty} f(k) = f(\infty)<\infty\}\,.
\end{equation*}
Let $E:\mathcal{T}\to\mathcal{T}_{\textrm{diag}}$ be the corresponding expectation defined by
\begin{equation*}
E(a) = \int_0^1 \rho_\theta(a)\,d\theta\,.
\end{equation*}
For any $a\in\mathcal{T}$, the Fourier series coefficients $a_(\K)$ are defined as:
\begin{equation*}
a_n(\K) = \left\{
\begin{aligned}
&E((U^*)^na) &&\textrm{ if }n\ge0 \\
&E(aU^{-n}) &&\textrm{ if }n<0\,.
\end{aligned}\right.
\end{equation*}
Without further assumptions the formal series representation given by \eqref{poly_U}) is, in general, not norm convergent.  However, the Fourier coefficients uniquely determine $a$ and if the series converges absolutely, then it converges to $a$ by the usual Fourier series arguments, see \cite{K} for details.    

It is known that $\mathcal{K}$ is isomorphic to the commutator ideal in $\mathcal{T}$, and $\mathcal{T}/\mathcal{K}\cong C(\R/\Z)$, see \cite{KD}.  Let ${q}:\mathcal{T}\to C(\R/\Z)$ be the corresponding quotient map.  

\subsection{Toeplitz Operators}
We identify $H$ with $\ell^2(\Z_{\ge0})$ through the basis $\{E_k\}$.  In turn, $\ell^2(\Z_{\ge0})$ is a closed subspace of $\ell^2(\Z)$.  Let $P:\ell^2(\Z)\to\ell^2(\Z_{\ge0})$ be the corresponding orthogonal projection.  If $f\in C(\R/\Z)$, consider the multiplication operator by $f$ in $L^2(\R/\Z)$.  Using Fourier series, we can naturally identify $L^2(\R/\Z)$ with $\ell^2(\Z)$, and thus let $M_f$ be the corresponding multiplication operator by $f$ in $\ell^2(\Z)$.

The Toeplitz operator $T(f)$ in $\ell^2(\Z_{\ge0})$ and hence in $H$, is defined by:
\begin{equation*}
T(f)\varphi = PM_f\varphi\quad\textrm{for}\quad \varphi\in\ell^2(\Z_{\ge0})\subseteq\ell^2(\Z)\,.
\end{equation*}
It is easy to see that
\begin{equation*}
T(e^{2\pi ix}) = U\quad\textrm{and}\quad T(e^{-2\pi ix}) = U^*\,.
\end{equation*}
Also, we have that $T(f)^* = T({\overline{f}})$ and that
\begin{equation}\label{Tfnorm}
\|T(f)\|\le \underset{x\in\R/\Z}{\textrm{sup}}|f(x)| = \|f\|_\infty\,.
\end{equation}
It follows that for every $f\in C(\R/\Z)$, one has $T(f)\in\mathcal{T}$.  Moreover, we have an important relation for the quotient map:
$${q} T(f) = f.$$

If $f\in C^\infty(\R/\Z)$, then $f$ has a convergent Fourier series given by:
\begin{equation}\label{fun_series}
f(x) = \sum_{n\in\Z} f_ne^{2\pi inx}\,.
\end{equation}
In fact, $f\in C^\infty(\R/\Z)$ if and only if $\{f_n\}_{n\in\Z}$ is a RD sequence.  It follows from the above that if $f\in C^\infty(\R/\Z)$, then $T(f)$ has the following norm convergent series representation:
\begin{equation*}
T(f) = \sum_{n\ge0} f_nU^n + \sum_{n<0}f_n(U^*)^{-n}\,.
\end{equation*}
Also notice that if $f\in C^\infty(\R/\Z)$, then $T(f)$ is in the domain of $\delta_\K$ since
\begin{equation}\label{delta_k_2}
\delta_\K(T(f)) = \sum_{n\ge0} nf_nU^n + \sum_{n<0}nf_n(U^*)^{-n} = T\left(\frac{1}{2\pi i}\frac{d}{dx}f\right)\,.
\end{equation}
Thus we call such Toeplitz operators to be smooth Toeplitz operators.  Additionally, we will need the fact that $T(f)\in\mathcal{K}$ if and only if $f=0$,  see \cite{KD} for details.

\subsection{Smooth Functions on the Circle}
We will use the following norms on the space of smooth functions on the circle $C^\infty(\R/\Z)$:
\begin{equation*}
\|f\|_{C^l}=\sum_{j=0}^l\begin{pmatrix} l\\j\end{pmatrix}\left\|\left(\frac{1}{2\pi i}\frac{d}{dx}\right)^jf\right\|_\infty.
\end{equation*}
The norms are submultiplicative and they satisfy the following inductive property:
\begin{equation}\label{Cl_inductive}
\|f\|_{C^{l+1}}=\|f\|_{C^l}+\left\|\frac{1}{2\pi i}\frac{d}{dx}\,f\right\|_{C^l}.
\end{equation}
Using Fourier series \eqref{fun_series} we have the following inequality:
\begin{equation*}
\|f\|_{C^l}\leq \sum_{j=0}^l\begin{pmatrix} l\\j\end{pmatrix}\| \sum_{n\in\Z}n^j f_ne^{2\pi inx}\|\leq \sum_{n\in\Z}|f_n|(1+|n|)^{l}\,.
\end{equation*}
We also have the other way inequality:
\begin{equation}\label{ftnorm_in}
\begin{aligned}
&\sum_{n\in\Z}|f_n|(1+|n|)^{l}\leq \sum_{j=0}^{l+2}\begin{pmatrix} l+2\\j\end{pmatrix}\sum_{n\in\Z}|n^jf_n|\frac{1}{(1+|n|)^2}=\\
&= \sum_{j=0}^{l+2}\begin{pmatrix} l+2\\j\end{pmatrix}\sum_{n\in\Z}\left|\int_0^1e^{-2\pi inx}\left(\frac{1}{2\pi i}\frac{d}{dx}\right)^jf(x)\,dx\right|\frac{1}{(1+|n|)^2}\leq\left(\frac{\pi^2}{3}-1\right)\|f\|_{C^{l+2}}\,.
\end{aligned}
\end{equation}
Thus, the norms $\sum_{n\in\Z}|f_n|(1+|n|)^{l}$ are equivalent to the norms $\|f\|_{C^l}$.

\subsection{Properties of Smooth Toeplitz Algebra}
We define $\mathcal{T}^\infty$ to be the subset of $\mathcal{T}$ consisting of sums of smooth Toeplitz operators and smooth compact operators, that is
\begin{equation*}
\mathcal{T}^\infty = \{a = T(f) + c: f\in C^\infty(\R/\Z), c\in\mathcal{K}^\infty\}\,.
\end{equation*}
Such a sum decomposition is unique since $T(f)\in \mathcal{K}^\infty$ if and only if $f=0$.  So, as a vector space, we have
\begin{equation*}
\mathcal{T}^\infty \cong C^\infty(\R/\Z) \oplus \mathcal{K}^\infty
\end{equation*}
and this has a structure of a Fr\'{e}chet space.  

In terms of Fourier series we have for an $a=T(f)+c\in\mathcal{T}^\infty$ that
\begin{equation*}
a= \sum_{n\ge0}U^n(f_n + c_n(\K)) + \sum_{n<0}(f_n + c_n(\K))(U^*)^{-n}\,,
\end{equation*}
and $\{f_n\}_{n\in\Z}$ is a RD sequence and $\{c_n(k)\}_{n\in\Z,k\ge0}$ is a RD sequence in both indices.  It is clear from this expansion that $\delta_\K$ is a continuous derivation in $\mathcal{T}^\infty$.  The main properties of the smooth Toeplitz algebra are summarized in the following proposition.

\begin{prop}\label{smooth_toep_props}
If $f\in C^\infty(\R/\Z)$ and $c\in\mathcal{K}^\infty$, then both $T(f)c$ and $cT(f)$ are in $\mathcal{K}^\infty$.  If $f$ and $g$ are in $C^\infty(\R/\Z)$, then
\begin{equation*}
T(f)T(g)- T(fg)\in\mathcal{K}^\infty\,.
\end{equation*}
Consequently, $\mathcal{T}^\infty$ is a $*$-subalgebra of $\mathcal{T}$ with $a^* = T({\overline{f}}) + c^*$ for a given $a\in\mathcal{T}^\infty$.  Moreover, $\mathcal{K}^\infty$ is a two-sided ideal of $\mathcal{T}^\infty$ with
\begin{equation*}
\mathcal{T}^\infty/\mathcal{K}^\infty\cong C^\infty(\R/\Z)\,,
\end{equation*}
and the quotient map ${q}$ acts as ${q}(T(f)+c) = f$ with $f\in C^\infty(\R/\Z)$.
\end{prop}

\begin{proof}
Notice that from Proposition \ref{N-norm_basics}, we have:
\begin{equation*}
\|T(f)c\|_{0,N}\le \|T(f)\|\|c\|_{0,N}<\infty\,.
\end{equation*}
Since $\delta_\K$ is a derivation we can write
\begin{equation*}
\delta_\K(T(f)c) = T\left(\frac{1}{2\pi i}\frac{d}{dx}f\right)c + T(f)\delta_\K(c)\,,
\end{equation*}
and in general
\begin{equation*}
\delta_\K^j(T(f)c) = \sum_ia_ib_i\,,\textrm{ finite sum,}
\end{equation*}
where the $a_i$ are bounded and the $b_i$ are smooth compact.  It therefore follows from Proposition \ref{N-norm_basics} that  $\|T(f)c\|_{M,N}<\infty$
and hence $T(f)c\in\mathcal{K}^\infty$.  In fact, we have by induction on $M$:
\begin{equation}\label{leftT(f)}
\|T(f)c\|_{M,N}\leq \|f\|_{C^M}\|c\|_{M,N}\,.
\end{equation}
The inductive step is:
\begin{equation*}
\begin{aligned}
&\|T(f)c\|_{M+1,N}=\|T(f)c\|_{M,N}+\|\delta_\K(T(f))c+T(f)\delta_\K(c)\|_{M,N}\leq \\
&\leq \left(\|f\|_{C^M}+\|f\|_{C^M}\right)(\|c\|_{M,N}+\|\delta_\K(c)\|_{M,N})=\|f\|_{C^{M+1}}\|c\|_{M+1,N}\,.
\end{aligned}
\end{equation*}

Since $\mathcal{K}^\infty$ is $*$-closed, we have that $cT(f)\in\mathcal{K}^\infty$.  A little more work gives an inequality:
\begin{equation}\label{rightT(f)}
\|cT(f)\|_{M,N}\leq \|f\|_{C^{M+N}}\|c\|_{M,N}\,.
\end{equation}
To prove it we proceed by induction on $M$. The case $M=0$ is the key and requires the bulk of the work. In estimating norms $\|cT(f)\|_{0,N}=\|cT(f)(I+\K)^M\|$ we pass to adjoints and then rewrite them in terms of derivations $\partial_N$ as follows:
\begin{equation*}
\|cT(f)(I+\K)^M\|=\|(I+\K)^MT(\bar f)c^*\|\leq \|\partial_N(T(\bar f))c^*\|+\|T(\bar f)(I+\K)^Mc^*\|\,.
\end{equation*}
Next, we use the identity \eqref{partial_delta} to express $\partial_N(T(\bar f))$ in terms of powers of $\delta_\K(T(\bar f))$:
\begin{equation*}
\|cT(f)\|_{0,N}\leq \sum_{j=0}^N\begin{pmatrix} N\\j\end{pmatrix}\|\delta_\K^j(T(\bar f))(I+\K)^{N-j}c^*\|\,.
\end{equation*}
Passing to adjoints again and using \eqref{delta_k_2} we can estimate as follows:
\begin{equation*}
\|cT(f)\|_{0,N}\leq \sum_{j=0}^N\begin{pmatrix} N\\j\end{pmatrix}\left\|T\left(\left(\frac{1}{2\pi i}\frac{d}{dx}\right)^jf\right)\right\|\|c\|_{0,N-j}\leq \|f\|_{C^{N}}\|c\|_{0,N}\,,
\end{equation*}
which concludes the case $M=0$. The inductive step is straightforward:
\begin{equation*}
\begin{aligned}
&\|cT(f)\|_{M+1,N}=\|cT(f)\|_{M,N}+\|\delta_K(c)T(f)+c\delta_\K(T(f))\|_{M,N}\leq\|c\|_{M,N}\|f\|_{C^{M+N}}+\\
&+\|\delta_\K(c)\|_{M,N}\|f\|_{C^{M+N}}+\|c\|_{M,N}\left\|\left(\frac{1}{2\pi i}\frac{d}{dx}\right)f\right\|_{C^{M+N}}\leq \|f\|_{C^{M+1+N}}\|c\|_{M+1,N}\,
\end{aligned}
\end{equation*}
which establishes \eqref{rightT(f)}.

We proceed to studying $T(f)T(g)- T(fg)$. Let $f$ and $g$ be in $C^\infty(\R/\Z)$ and write them in their respective Fourier decomposition, that is
\begin{equation*}
f(x) = \sum_{n\in\Z}f_ne^{2\pi inx}\quad\textrm{and}\quad g(x) = \sum_{n\in\Z}g_ne^{2\pi inx}\,.
\end{equation*}
Split $f$ into its positive and negative frequencies: $f= f_+ + f_-$ where
\begin{equation*}
f_+(x) = \sum_{n\ge0}f_ne^{2\pi inx}\quad\textrm{and}\quad f_-(x) = \sum_{n<0}f_ne^{2\pi inx}\,.
\end{equation*}
Also split $g$ in the same fashion.  A direct calculation shows that
\begin{equation*}
T(f_+)T(g_+) = T(f_+g_+)\,,\,\,T(f_-)T(g_-) = T(f_-g_-)\,,\,\,T(g_-)T(f_+) = T(g_-f_+)\,.
\end{equation*}
Using these relations, it follows that
\begin{equation*}
T(f)T(g) - T(fg) = T(f_+)T(g_-) - T(f_+g_-) = T(f_+)T(g_-) - T(g_-)T(f_+)\,.
\end{equation*}
For $n<0$, let $P_{<-n}$ be the orthogonal projection in $H$ onto the subspace spanned by $\{E_k\}_{k<-n}$.  Using the fact that
\begin{equation*}
U^{-n}(U^*)^{-n} = I - P_{<-n}
\end{equation*}
for $n<0$ and expanding $T(g_-)$ as
\begin{equation*}
T(g_-) = \sum_{n<0}g_n(U^*)^{-n}\,,
\end{equation*}
we obtain that
\begin{equation*}
T(f)T(g)-T(fg) = -\sum_{n<0}g_n(U^*)^{-n}T(f_+)P_{<-n}\,.
\end{equation*}
Applying $\delta_\K$ to $T(f)T(g)-T(fg)$ $l$-times and using the formula \eqref{delta_k_1} for $\delta_\K(U^*)$ and the formula \eqref{delta_k_2} for $\delta_\K(T(f))$, we get
\begin{equation*}
\delta_\K^l(T(f)T(g)-T(fg)) = -\sum_{n<0}\sum_{j=0}^l\begin{pmatrix} l\\j\end{pmatrix}n^{l-j}g_n(U^*)^{-n}T\left(\left(\frac{1}{2\pi i}\frac{d}{dx}\right)^jf_+\right)P_{<-n}\,.
\end{equation*}
To estimate the norms we first notice that
\begin{equation*}
\|P_{<-n}\|_{0,N} =|n|^N\,.
\end{equation*}
Consequently, we have
\begin{equation*}
\begin{aligned}
\|\delta_\K^l(T(f)T(g)-T(fg))\|_{0,N} &\le\sum_{n<0}\sum_{j=0}^l|g_n|\,|n|^{l-j+N}\begin{pmatrix} l\\j\end{pmatrix}\left\|\left(\frac{1}{2\pi i}\frac{d}{dx}\right)^jf_+\right\| \\
&\le\sum_{n<0}|g_n|(1+|n|)^{l+N}\|f\|_{C^l}\le \left(\frac{\pi^2}{3}-1\right)\|g\|_{C^{l+N+2}}\|f\|_{C^l}<\infty\,,
\end{aligned}
\end{equation*}
from estimate \eqref{ftnorm_in}.
Therefore, we finally get $\|T(f)T(g)-T(fg)\|_{M,N}<\infty$. In fact, with a little more work, we have an estimate:
\begin{equation}\label{TProd_estimate}
\|T(f)T(g)-T(fg)\|_{M,N}\leq \left(\frac{\pi^2}{3}-1\right)\|g\|_{C^{M+N+2}}\|f\|_{C^M} \quad\textrm{for all}\quad M,N\,.
\end{equation}
The key step is contained in the following bound that uses a binomial coefficients inequality:
\begin{equation*}\begin{aligned}
&\sum_{l=0}^M\sum_{j=0}^l\begin{pmatrix} M\\l\end{pmatrix}\begin{pmatrix} l\\j\end{pmatrix}|n|^{l-j}\left\|\left(\frac{1}{2\pi i}\frac{d}{dx}\right)^jf\right\| \leq \sum_{l=0}^M\sum_{j=0}^l\begin{pmatrix} M\\l-j\end{pmatrix}\begin{pmatrix} M\\j\end{pmatrix}|n|^{l-j}\left\|\left(\frac{1}{2\pi i}\frac{d}{dx}\right)^jf\right\|\\
&\leq \sum_{j=0}^M\sum_{k=0}^M\begin{pmatrix} M\\k\end{pmatrix}\begin{pmatrix} M\\j\end{pmatrix}|n|^{k}\left\|\left(\frac{1}{2\pi i}\frac{d}{dx}\right)^jf\right\| = (1+|n|)^M\|f\|_{C^M}\,.
\end{aligned}
\end{equation*}
\end{proof}

We notice here for future use, that the proof above implies that the maps $c\mapsto T(f)c$ and $c\mapsto cT(f)$ are continuous maps of $\mathcal{K}^\infty$ for $f\in C^\infty(\R/\Z)$.  Compare this to our previous discussion that the maps $c\mapsto c\K$ and $c\mapsto \K c$ are also continuous maps of $\mathcal{K}^\infty$.  This proof also demonstrates that given a pair $(f,g)\in C^\infty(\R/\Z)\times C^\infty(\R/\Z)$ the map
\begin{equation*}
(f,g)\mapsto T(f)T(g)-T(fg)\in\mathcal{K}^\infty
\end{equation*}
is jointly continuous.

Choose a constant $S$ such that $S^2\geq S+ \frac{\pi^2}{3}-1$. The smallest such $S$ is
\begin{equation*}
S=\frac12+\sqrt{\frac{4\pi^2-9}{12}}\,.
\end{equation*}
Using the above analysis we can now construct submultiplicative norms on $\mathcal{T}^\infty$.
\begin{prop}
The Toeplitz algebra $\mathcal{T}^\infty$ is a $*$-Fr\'{e}chet algebra with respect to the following submultiplicative norms:
\begin{equation*}
\|T(f)+c\|_{M,N}:=S\|f\|_{C^{M+N+2}}+\|c\|_{M,N}\,.
\end{equation*}
\end{prop}
\begin{proof}
We only need to address the submultiplicativity of the above norms. We write the product of two elements of $\mathcal{T}^\infty$ as:
\begin{equation*}
(T(f_1)+c_1)(T(f_2)+c_2)=T(f_1f_2)+\left((T(f_1)T(f_2)-T(f_1f_2))+c_1T(f_2)+T(f_1)c_2+c_1c_2\right).
\end{equation*}
The submultiplicativity of the norms now easily follows using inequalities \eqref{leftT(f)}, \eqref{rightT(f)}, \eqref{TProd_estimate}, submultiplicativity of $\|\cdot\|_{M,N}$ on $\mathcal{K}^\infty$, submultiplicativity of norms $\|\cdot\|_{C^{M}}$ and the definition of $S$.
\end{proof}

\subsection{Stability of Smooth Toeplitz Algebra}
The last properties of $\mathcal{T}^\infty$ discussed in this section are the stability under the holomorphic calculus and the stability under smooth calculus of self-adjoint elements. The first stability implies that $\mathcal{T}^\infty$ has the same $K$-Theory as $\mathcal{T}$, see \cite{Bo}. 
\begin{theo}\label{ToepStableProp}
The smooth Toeplitz algebra $\mathcal{T}^\infty$ is closed under the holomorphic functional calculus. In other words, for any $a\in \mathcal{T}^\infty$ and a function $f$ that is holomorphic on an open domain containing the spectrum of $a$ we have $f(c)\in \mathcal{T}^\infty$.
\end{theo}

\begin{proof}
Since $\mathcal{T}^\infty$ is a Fr\'{e}chet space, it is enough to check that if $a\in\mathcal{T}^\infty$ and invertible in $\mathcal{T}$, then $a^{-1}\in\mathcal{T}^\infty$.  The statement will then follow from the usual Cauchy integral representation.  

Let $a\in\mathcal{T}^\infty$, then $a=T(f) +c$ with $f\in C^\infty(\R/\Z)$ and $c\in\mathcal{K}^\infty$ and suppose $a$ is invertible in $\mathcal{T}$.  Since the quotient map ${q}$ is a homomorphism, ${q}(a)=f$ is invertible in $C(\R/\Z)$, that is $f(x)\neq0$ for all $x\in\R/\Z$.  However, since $f$ is smooth, it follows that $1/f$ is smooth.  Therefore we have
\begin{equation*}
a^{-1} = T(1/f) + b
\end{equation*}
for some $b\in\mathcal{K}$.  The proof will be complete if we can show that $b\in\mathcal{K}^\infty$.  Notice that
\begin{equation*}
b = a^{-1} - T(1/f) = a^{-1}(I-aT(1/f)) = a^{-1}(I - T(f)T(1/f) + cT(1/f))\,.
\end{equation*}
From Proposition \ref{smooth_toep_props}, we have that both $1-T(f)T(1/f)$ and $cT(1/f)$ are in $\mathcal{K}^\infty$.  Consequently, there is a $\tilde{c}\in\mathcal{K}^\infty$ such that $b=a^{-1}\tilde{c}$.  It follows from Proposition \ref{N-norm_basics} that
\begin{equation}\label{inverse_est}
\|b\|_{0,N}\le\|a^{-1}\|\|\tilde{c}\|_{0,N}<\infty\,.
\end{equation}
Computing $\delta_\K$ on $b$ we have
\begin{equation*}
\delta_\K(b) = \delta_\K(a^{-1})\tilde{c} + a^{-1}\delta_\K(\tilde{c}) = -a^{-1}\delta_\K(a)a^{-1}\tilde{c} + a^{-1}\delta_\K(\tilde{c})\,.
\end{equation*}
So, as in the proofs of Proposition \ref{invert_K_infty} or in Proposition \ref{smooth_toep_props}, we have, inductively, for any $j$ that
\begin{equation*}
\delta_\K^j(b) = \sum_ia_ib_i \quad\textrm{finite sum,}
\end{equation*}
with $a_i$ bounded and $b_i$ are smooth compact.  Using this and the estimate in equation \eqref{inverse_est}, we see that $\|b\|_{M,N}$ is finite for all $M$ and $N$.  Thus $b\in\mathcal{K}^\infty$, completing the proof.
\end{proof}

Next we will discuss stability of $\mathcal{T}^\infty$ under the smooth functional calculus of self-adjoint elements. Similarly to the discussion of stability of $\mathcal{K}^\infty$, this will require controlling growth of norms of exponentials of elements of $\mathcal{T}^\infty$.

\begin{prop}\label{exp_of_f}
If $f\in C^\infty(\R/\Z)$ is real, $f=\bar f$, then we have an estimate:
\begin{equation*}
\|e^{if}\|_{C^M} \leq \prod_{j=1}^M (1+\|f\|_{C^j})\,.
\end{equation*}
\end{prop}
\begin{proof} 
Clearly $\|e^{if}\|_{C^0}=1$. For $M\geq 1$ the inequality is a simple consequence of an inductive argument using formula \eqref{Cl_inductive}.
\end{proof}

\begin{prop}\label{exp_of_c}
Suppose that $c\in\mathcal{K}^\infty$ is a self-adjoint smooth compact operator. Then we have an estimate:
\begin{equation*}
\|e^{ic}\|_{M,0}\leq \prod_{j=1}^M (1+\|c\|_{j,0})^{2^{M-j}}\,.
\end{equation*}
\end{prop}
\begin{proof} 
For $M=0$ we have $\|e^{ic}\|_{0,0}=1$. We proceed by induction on $M$ utilizing the relation:
\begin{equation*}
\|e^{ic}\|_{M+1,0} = \|e^{ic}\|_{M,0} + \|\delta_\K(e^{ic})\|_{M,0}\,.
\end{equation*}
Using \eqref{duhamel} for $j=1$ we have
\begin{equation*}
\delta_\K(e^{ic})=i\int_0^1e^{i(1-t)c}\delta_\K(c)e^{itc}\,dt\,.
\end{equation*}
Then the inductive step estimate is:
\begin{equation*}
\begin{aligned}
&\|e^{ic}\|_{M+1,0}\leq \|e^{ic}\|_{M,0}+i\int_0^1\|e^{i(1-t)c}\|_{M,0}\|\delta_\K(c)\|_{M,0}\|e^{itc}\|_{M,0}\,dt\leq\\
&\leq \prod_{j=1}^M (1+\|c\|_{j,0})^{2^{M-j}} +\left[\prod_{j=1}^M (1+\|c\|_{j,0})^{2^{M-j}}\right]^2\|\delta_\K(c)\|_{M,0}\,.
\end{aligned}
\end{equation*}
Since $\|\delta_\K(c)\|_{M,0}\leq\|c\|_{M+1,0}$, we have:
\begin{equation*}
\begin{aligned}
&\|e^{ic}\|_{M+1,0}\leq \prod_{j=1}^M (1+\|c\|_{j,0})^{2^{M-j}}(1+\prod_{j=1}^M (1+\|c\|_{j,0})^{2^{M-j}}\|c\|_{M+1,0})\leq\\
&\leq \prod_{j=1}^M (1+\|c\|_{j,0})^{2^{M-j}}\prod_{j=1}^M (1+\|c\|_{j,0})^{2^{M-j}}(1+\|c\|_{M+1,0})=\prod_{j=1}^{M+1} (1+\|c\|_{j,0})^{2^{M+1-j}}\,.
\end{aligned}
\end{equation*}
This establishes the inductive step and finishes the proof.
\end{proof}

\begin{theo}
The smooth Toeplitz algebra $\mathcal{T}^\infty$ is closed under the smooth functional calculus of self-adjoint elements. In other words, if $a\in\mathcal{T}^\infty$ is self-adjoint and $f(x)$ is a smooth function on an open neighborhood of the spectrum $\sigma(a)$ of $a$ then $f(a)$ is in $\mathcal{T}^\infty$.
\end{theo}
\begin{proof} Proceeding as in the proof of smooth stability of smooth compact operators we may assume that $f(x)$ is smooth on $\R$ and is $L$-periodic: $f(x+L)=f(x)$ for some $L$. Then $f(x)$ admits a Fourier series representation with rapid decay coefficients $\{f_n\}$ and we have:
\begin{equation*}
f(a)=\sum_{n\in\Z}f_ne^{2\pi ina/L}
\end{equation*}
for a self-adjoint $a=T(f)+c\in\mathcal{T}^\infty$. By submultiplicativity of the $\|\cdot\|_{M,N}$ on $\mathcal{T}^\infty$, the exponentials $e^{2\pi ina/L}$ are in $\mathcal{T}^\infty$. Thus, the theorem will be proved if we can establish at most polynomial growth in $n$ of norms $\|e^{2\pi ina/L}\|_{M,N}$.

Notice that the symbol ${q}\left(e^{2\pi ina/L}\right)$ in $C^\infty(\R/\Z)$ is $e^{2\pi inf/L}$, which, by Proposition \ref{exp_of_f} grows at most polynomially in $n$. Thus, we are reduced to establishing that the $\|\cdot\|_{M,N}$ of the difference
\begin{equation*}
e^{2\pi in(T(f)+c)/L}-T\left(e^{2\pi inf/L}\right)\in\mathcal{K}^\infty
\end{equation*}
are at most polynomially growing in $n$. 

To analyze the above expression we employ a version of the Duhamel's formula:
\begin{equation*}
\begin{aligned}
&e^{i(T(f)+c)}-T\left(e^{if}\right)=\int_0^1\frac{d}{dt}\left( e^{it(T(f)+c)}T\left(e^{i(1-t)f}\right)\right)dt=\\
&=\int_0^1 e^{it(T(f)+c)}c\,T\left(e^{i(1-t)f}\right)dt +\int_0^1 e^{it(T(f)+c)}\left[T(f)T\left(e^{i(1-t)f}\right)-T\left(fe^{i(1-t)f}\right)\right]dt\,.
\end{aligned}
\end{equation*}
Using Proposition \ref{MN-norm_basics} we can estimate the norms of those exponential as follows:
\begin{equation*}
\begin{aligned}
&\|e^{i(T(f)+c)}-T\left(e^{if}\right)\|_{M,N}\leq \int_0^1 \|e^{it(T(f)+c)}\|_{M,0}\|c\,T\left(e^{i(1-t)f}\right)\|_{M,N}\,dt+\\
& +\int_0^1 \|e^{it(T(f)+c)}\|_{M,0}\|T(f)T\left(e^{i(1-t)f}\right)-T\left(fe^{i(1-t)f}\right)\|_{M,N}\,dt\,.
\end{aligned}
\end{equation*}
All the terms in the above expression can now be estimated using \eqref{ftnorm_in}, \eqref{rightT(f)} and Propositions \ref{exp_of_f}, \ref{exp_of_c}.
We obtain the following bounds:
\begin{equation*}
\begin{aligned}
&\|e^{i(T(f)+c)}-T\left(e^{if}\right)\|_{M,N}\leq \prod_{j=1}^M (1+\|f\|_{C^j}+\|c\|_{j,0})^{2^{M-j}}\,\|c\|_{M,N}\prod_{j=1}^{M+N} (1+\|f\|_{C^j})+\\
& + \left(\frac{\pi^2}{3}-1\right) \prod_{j=1}^M (1+\|f\|_{C^j}+\|c\|_{j,0})^{2^{M-j}}\, \|f\|_{C^M} \prod_{j=1}^{M+N+2} (1+\|f\|_{C^j})\,.
\end{aligned}
\end{equation*}
Clearly those estimates establish the desired at most polynomial growth, finishing the proof.
\end{proof}

\section{Structure of Derivations}
\subsection{Basics}
Let $A$ be an algebra. A linear map $\delta:A\to A$ is called a {\it derivation} if the Leibniz rule holds:
\begin{equation*}
\delta(ab) = a\delta(b) + \delta(a)b
\end{equation*}
In this paper we consider three C$^*$-algebras: commutative C$^*$-algebra $C(\R/\Z)$, the algebra of compact operators $\mathcal{K}$ and the Toeplitz algebra $\mathcal{T}$. In each of those algebras we singled out a smooth subalgebra $C^\infty(\R/\Z)$, $\mathcal{K}^\infty$ and $\mathcal{T}^\infty$. The purpose of this section is to discuss derivations on all 6 algebras. Remarkably, there are significant differences in spaces of derivations on C$^*$-algebras and on their smooth subalgebras.

Theory of derivations on operator algebras is well-developed, see for example \cite{KR}, \cite{S}. It turns out that a derivation on a C$^*$-algebra is automatically continuous.  While in general this property does not hold for derivations on Frechet algebras, it however remains true for $C^\infty(\R/\Z)$ and $\mathcal{T}^\infty$, chiefly because they are finitely generated.

We say that a derivation $\delta$ of a C$^*$-algebra $A$ of operators on a Hilbert space $H$ is spatial when there is a bounded operator $b$ on $H$ such that $\delta (a)= ba - ab$, for each $a$ in $A$. If $b$ can be chosen in $A$, we say that $\delta$ is inner.
It turns out that all derivations of von Neumann algebras are inner. Moreover, each derivation of a C$^*$-algebra is spatial, induced by an operator in the weak closure of $A$ in $H$.
In particular, derivations of commutative C$^*$-algebras, including $C(\R/\Z)$, are 0. Furthermore, because $\mathcal{K}$ is an ideal in the algebra of bounded operators, any derivation on $\mathcal{K}$ is spatial and any bounded operator induces a derivation on $\mathcal{K}$.
The space of derivations on the Toeplitz algebra is more subtle and it coincides with the essential commutator algebra of the unilateral shift \cite{BH}.

A standard result in differential geometry (see \cite{NW}) is that for a smooth manifold $M$ derivations on $C^\infty(M)$ are precisely smooth vector fields on $M$.
In particular, any derivation on $C^\infty(\R/\Z)$ is of the form:
\begin{equation*}
F(x)\frac{1}{2\pi i}\frac{d}{dx}
\end{equation*}
and is automatically continuous.

It remains to describe derivations on $\mathcal{K}^\infty$ and on $\mathcal{T}^\infty$. This is done in the following sections.
\subsection{Derivations on $\mathcal{K}^\infty$}
For completeness we classify here all continuous derivations $\delta$ on $\mathcal{K}^\infty$. We follow the strategy used in our previous papers \cite{KMRSW1}, \cite{KMRSW2}, decomposing derivations into Fourier components which satisfy covariance properties and then classifying continuous covariant derivations. The result is not particularly elegant: continuous derivations on $\mathcal{K}^\infty$ are given by commutators with possibly unbounded operators whose matrix coefficients satisfy growth conditions detailed in Theorem \ref{compact_der_theo} below.

Given $n\in\Z$, a derivation $\delta:\mathcal{K}^\infty\to \mathcal{K}^\infty$ is said to be a {\it $n$-covariant derivation} if the relation 
\begin{equation*}
\rho_\theta^{-1}\delta\rho_\theta(a)= e^{-2\pi in\theta} \delta(a)
\end{equation*}
holds for all $\theta$, where $\rho_\theta$ was defined in \eqref{rho_def}.  When $n=0$ we say that the derivation is invariant.  With this definition, we point out that $\delta_\K:\mathcal{K}^\infty\to \mathcal{K}^\infty$ is an invariant continuous derivation.

\begin{defin}
If $\delta$ is a continuous derivation in $\mathcal{K}^\infty$, the {\it $n$-th Fourier component} of $\delta$ is defined as: 
\begin{equation*}
\delta_n(a)= \int_0^1 e^{2\pi in\theta} \rho_\theta^{-1}\delta\rho_\theta(a)\, d\theta\,.
\end{equation*}
\end{defin}
Below we describe properties of Fourier components of derivations and of $n$-covariant derivations in general.
\begin{prop}
Let $\delta:\mathcal{K}^\infty\to \mathcal{K}^\infty$ be a continuous derivation.  Then $\delta_n:\mathcal{K}^\infty\to \mathcal{K}^\infty$ is a continuous $n$-covariant derivation, where $\delta_n$ are the $n$-th Fourier components of $\delta$.
\end{prop}
\begin{proof}
It is straightforward to see that $\delta_n$ is a derivation and is well-defined on $\mathcal{K}^\infty$. 
Since $\delta:\mathcal{K}^\infty\to \mathcal{K}^\infty$, $\rho_\theta:\mathcal{K}^\infty\to\mathcal{K}^\infty$ and $\mathcal{K}^\infty$ is complete, it follows that $\delta_n(a)\in \mathcal{K}^\infty$ for all $a\in\mathcal{K}^\infty$. Since $\delta$ is a continuous derivation and the automorphism $\rho_\theta$ is continuous, it follows that $\delta_n$ is also continuous, see also the proof of the lemma below.

The following computation verifies that $\delta_n$ is $n$-covariant:
\begin{equation*}
\rho_\theta^{-1}\delta_n\rho_\theta(a) = \int_0^1 e^{2\pi in\varphi} \rho_\theta^{-1}\rho_\varphi^{-1}\delta\rho_\varphi\rho_\theta(a)\, d\varphi = \int_0^1 e^{2\pi in\varphi}\rho_{\theta + \varphi}^{-1}\delta\rho_{\theta + \varphi}(a)\, d\varphi\,.
\end{equation*}
Changing to new variable $\theta + \varphi$, and using the translation invariance of the measure, it now follows that $\rho_\theta^{-1}\delta_n\rho_\theta(a)= e^{-2\pi in\theta} \delta_n(a)$.    
\end{proof}
We will also need the following growth estimate on Fourier components of derivations.
\begin{lem}\label{est_delta_n}
Let $\delta:\mathcal{K}^\infty\to\mathcal{K}^\infty$ be a continuous derivation.  Then, for every $k\ge0$, $M\ge0$ and $N\ge0$, there exist $M'\ge0$ and $N'\ge0$ and a constant $C_k=C_k(M,N)$ such that for every $a\in \mathcal{K}^\infty$ we have:
\begin{equation*}
n^k\|\delta_n(a)\|_{M,N} \le C_k\|a\|_{M',N'}\,,
\end{equation*}
where $\delta_n$ are the $n$-th Fourier components of $\delta$.
\end{lem}

\begin{proof}
Since $\delta$ is a continuous linear map on a Frechet space $\mathcal{K}^\infty$, there exists a constant $C=C(M,N)$ and an $M'\ge0$ and $N'\ge0$ such that for every $a\in \mathcal{K}^\infty$:
\begin{equation*}
\|\delta(a)\|_{M,N} \le C\|a\|_{M',N'}\,.
\end{equation*}
Using this and the fact that $\rho_\theta$ is continuous, we have
\begin{equation}\label{der_comp_cont}
\|\delta_n(a)\|_{M,N}= \left\|\int_0^1e^{2\pi in\theta}\rho_\theta^{-1}\delta\rho_\theta(a)\,d\theta\right\|_{M,N}\le \int_0^1\|\rho_\theta^{-1}\delta\rho_\theta(a)\|_{M,N}\,d\theta \le C\|a\|_{M',N'}\,.
\end{equation}
Consider the following calculation, using integration by parts, the continuity of $[\delta_\K,\delta]$, and the continuous differentiability of $\theta\mapsto \rho_\theta^{-1}\delta\rho_\theta$:
\begin{equation*}
\begin{aligned}
2\pi in\delta_n(a) &= \int_0^1 2\pi ine^{2\pi in\theta}\rho_\theta^{-1}\delta\rho_\theta(a)\,d\theta = \int_0^1\frac{d}{d\theta}\left(e^{2\pi in\theta}\right)\rho_\theta^{-1}\delta\rho_\theta(a)\,d\theta \\
&=-\int_0^1e^{2\pi in\theta}\frac{d}{d\theta}\left(\rho_\theta^{-1}\delta\rho_\theta(a)\right)\,d\theta=-2\pi i\int_0^1\rho_\theta^{-1}[\delta_\K,\delta]\rho_\theta(a)\,d\theta \\
&=-2\pi i\left([\delta_\K,\delta]\right)_n(a)\,,
\end{aligned}
\end{equation*}
where $\left([\delta_\K,\delta]\right)_n$ is the $n$-th Fourier component of the derivation $[\delta_\K,\delta]$. Consequently, using estimate \eqref{der_comp_cont} for the commutator $[\delta_\K,\delta]$ we obtain:
\begin{equation*}
n\|\delta_n(a)\|_{M,N} \le C_k\|a\|_{M',N'}\,,
\end{equation*}
for some $M'$, $N'$ and a constant $C$.
The result now easily follows by induction on $k$.
\end{proof}

The usual Ces\`aro mean convergence result for Fourier components in harmonic analysis \cite{K} implies that if $\delta$ is a continuous derivation onn $\mathcal{K}$ then
\begin{equation*}
\delta(a)=\lim_{L\rightarrow\infty} \frac{1}{L+1} \sum_{j=0}^L \left(\sum_{n=-j}^j \delta_n(a)\right)\,,
\end{equation*}
for every $a\in\mathcal{K}^\infty$.
In particular, $\delta$ is completely determined by its Fourier components $\delta_n$.

\begin{prop}\label{Fourier_series_norm_conv}
Let $\delta: \mathcal{K}^\infty\to \mathcal{K}^\infty$ be a continuous derivation, then
\begin{equation}\label{der_series}
\delta(a) = \sum_{n\in\Z} \delta_n(a)
\end{equation}
for every $a\in \mathcal{K}^\infty$ where $\delta_n$ are the $n$-th Fourier components of $\delta$.  Here the sum is norm convergent.
\end{prop}
\begin{proof}
Lemma \ref{est_delta_n} implies that $\{\|\delta_n(a)\|\}$ is a RD sequence for every $a\in\mathcal{K}^\infty$ and thus the series \eqref{der_series} is norm convergent. Also, the right-hand side of \eqref{der_series} is a continuous derivation on $\mathcal{K}^\infty$ with the same Fourier components $\delta_n$ as $\delta$ so they must be equal.

\end{proof}  

\begin{prop}\label{n_covariant_der_smooth_compact}
Let $\delta:\mathcal{K}^\infty\to \mathcal{K}^\infty$ be a $n$-covariant continuous derivation. Then there exists a polynomially bounded sequence $\{\beta_j\}$ such that $\delta(a)$ is given by the formal sum:
\begin{equation}\label{CovCompact}
\delta(a) = \left\{
\begin{aligned}
&\left[\sum_{j=0}^\infty \beta_jP_{j,j+n},a\right] &&\textrm{if }n\ge0\\
&\left[\sum_{j=0}^\infty \beta_jP_{j-n,j},a\right] &&\textrm{if }n<0
\end{aligned}\right.
\end{equation}
for all $a\in\mathcal{K}^\infty$. Conversely, a polynomially bounded sequence $\{\beta_j\}$ uniquely determines a $n$-covariant continuous derivation on $\mathcal{K}^\infty$ via the equation \eqref{CovCompact}, with the exception of $n=0$, when the sequence $\{\beta_j\}$ is determined up to an additive constant.
\end{prop} 
Here polynomially bounded means that there exist an $N\ge0$ and a constant $C\ge0$ so that for every $j$ we have:
\begin{equation*}
|\beta_j|\le C(1+j)^N\,.
\end{equation*}
\begin{proof}
We prove the decomposition formula and the growth condition estimate for $n\ge0$ as the case $n<0$ is very similar.   We first proceed with $n=0$ case and then consider $n>0$.  

Let $n=0$ and let $\mathcal{K}_{\textrm{diag}}^\infty$ be the fixed point algebra of $\rho_\theta$ in $\mathcal{K}^\infty$. That is, $a\in\mathcal{K}_{\textrm{diag}}^\infty$ if and only if $a\in\mathcal{K}^\infty$ and $\rho_\theta(a) = a$.  
Applying the automorphism $\rho_\theta$ to $P_{k,l}$ we have
\begin{equation*}
\rho_\theta(P_{k,l}) = e^{2\pi i(l-k)\theta}P_{k,l}\,.
\end{equation*}
It easily follows from the above formula that $\mathcal{K}_{\textrm{diag}}^\infty$ consists of diagonal operators in $\mathcal{K}^\infty$:
\begin{equation*}
\mathcal{K}_{\textrm{diag}}^\infty = \left\{a=\sum_{k\ge0}a_{k,k}P_{k,k}: \{a_{k,k}\}\textrm{ is a RD sequence}\right\}\,.
\end{equation*}
Notice that if $\delta$ is an invariant derivation, then it follows that $\delta:\mathcal{K}_{\textrm{diag}}^\infty\to\mathcal{K}_{\textrm{diag}}^\infty$.  
But $\mathcal{K}_{\textrm{diag}}^\infty$ is generated by projections $P_{k,k}$ and we have:
\begin{equation*}
\delta(P_{k,k})=\delta\left(P_{k,k}^2\right)=2P_{k,k}\delta(P_{k,k}),
\end{equation*}
so that $(1-2P_{k,k})\delta(P_{k,k})=0$, implying that $\delta(P_{k,k})=0$.
Therefore, by continuity, we have that $\delta(a)=0$ for all $a\in\mathcal{K}_{\textrm{diag}}^\infty$.  

Notice that
\begin{equation*}
P_{kk}P_{k,l} = P_{k,l}\quad\textrm{and}\quad P_{k,l}P_{l,l}=P_{k,l}\,.
\end{equation*}
Thus, applying $\delta$ to these relations and using the Leibniz rule we have
\begin{equation*}
P_{kk}\delta(P_{k,l}) = \delta(P_{k,l})\quad\textrm{and}\quad\delta(P_{k,l})P_{l,l} = \delta(P_{k,l}) 
\end{equation*}
since $P_{k,k}$ and $P_{l,l}$ belong to $\mathcal{K}_{\textrm{diag}}^\infty$.  Thus, there exists a sequence $\{b_{k,l}\}$ such that
\begin{equation*}
\delta(P_{k,l}) = b_{kl}P_{k,l}
\end{equation*}
with $b_{k,k}=0$. By continuity, $\delta$ is completely determined by the sequence $\{b_{k,l}\}$. To proceed further we use the relation
\begin{equation*}
P_{k,l}P_{i,j} = \chi_{l,i}P_{k,j}
\end{equation*}
with $\chi_{l,i} =1$ for $l=i$ and is equal to $0$ otherwise.  Applying $\delta$ to this relation and using the Leibniz rule again we have
\begin{equation*}
b_{k,l}P_{k,l}P_{i,j} + b_{i,j}P_{k,l}P_{i,j} = b_{k,j}\chi_{l,i}P_{k,j}\,.
\end{equation*}
Therefore, we can conclude that
\begin{equation*}
(b_{k,l} + b_{i,j})\chi_{l,i}P_{k,j} = b_{kj}\chi_{l,i}P_{k,j}\,.
\end{equation*}
So, if $l=i$ we have that
\begin{equation*}
b_{k,l} = b_{k,j}-b_{l,j}\quad\textrm{for all }j\,,
\end{equation*}
and hence there exists a sequence $\{\beta_k\}$, for example $\{\beta_k\}=b_{k,0}$, such that
\begin{equation}\label{inv_beta}
\delta(P_{k,l}) = (\beta_k-\beta_l)P_{k,l}\,.
\end{equation}
Consequently we have formally that
\begin{equation*}
\delta(a) = \left[\sum_{k\ge0}\beta_kP_{k,k},a\right]\,.
\end{equation*}
Clearly, by \eqref{inv_beta}, the sequence $\beta_j$ is determined up to additive constant, while polynomial boundedness of $\{\beta_k\}$ is discussed below.

Next suppose that $n>0$.   Then the covariance condition on $\delta$ implies that we have
\begin{equation*}
\rho_\theta(\delta(P_{k,l})) = e^{2\pi i(l-k+n)\theta}\delta(P_{k,l})\,,
\end{equation*}
so it follows that $\delta(P_{k,l})$ must have the following general form:
\begin{equation}\label{gen_n_der_proj}
\delta(P_{k,l}) = \sum_{s-r=l-k+n}b_{r,s}(k,l)P_{r,s}\,,
\end{equation}
where the summation above is over those $r\geq0$, $s\geq0$ such that $s-r=l-k+n$.
We want to analyze the coefficients $b_{r,s}(k,l)$.

Applying $\delta$ to the the relation:
\begin{equation*}
P_{k,l}P_{i,j} = \chi_{l,i}P_{k,j}
\end{equation*}
and using the Leibniz rule again we have
\begin{equation*}
\delta(P_{k,l})P_{i,j} + P_{k,l}\delta(P_{i,j}) = \chi_{l,i}\delta(P_{k,j})\,.
\end{equation*}
Substituting in equation \eqref{gen_n_der_proj} into the above equation yields the following:
\begin{equation*}
\sum_{s-r=l-k+n}b_{r,s}(k,l)P_{r,s}P_{i,j} + \sum_{s-r=l-k+n}b_{r,s}(k,l)P_{k,l}P_{r,s} = \chi_{l,i}\delta(P_{k,j})
\end{equation*}
We will use a useful convention that if one of the subscripts $r,s$ is negative then we set $b_{r,s}(k,l)=0$.
With this convention, the above equation reduces to
\begin{equation}\label{gen_seq_eq}
b_{i-l+k-n,i}(k,l)P_{i-l+k-n,j} + b_{l,l+j-i+n}(i,j)P_{k,l+j-i+n} = \chi_{l,i}\delta(P_{k,j})\,.
\end{equation}
There are several cases to consider. 

 First, if $l\neq i$, then the right-hand side of \eqref{gen_seq_eq} is zero. If additionally $i-l-n=0$, then equation \eqref{gen_seq_eq} reduces to
\begin{equation*}
(b_{k,l+n}(k,l)+b_{l,j}(l+n,j))P_{k,j}=0\quad\textrm{for all }j,\, k\,\textrm{ and }l\,.
\end{equation*}
In other words,
\begin{equation*}
b_{k,l+n}(k,l)= - b_{l,j}(l+n,j)\quad\textrm{for all }j,\, k\,\textrm{ and }l\,.
\end{equation*}
Notice that the left-hand side of the above equation does not depend on $j$ while the right-hand side does not depend on $k$.
It now follows that if we define $\{\beta_l\}$ as 
\begin{equation*}
\beta_l=b_{l,0}(l+n,0),
\end{equation*}
then we have that
\begin{equation*}
\beta_l=b_{l,0}(l+n,0)=b_{l,j}(l+n,j)=-b_{k,l+n}(k,l)\quad\textrm{for all }j,\, k\,\textrm{ and }l\,.
\end{equation*}

Next, notice that if $l=i$ we arrive at
\begin{equation*}
\delta(P_{k,j}) = b_{k-n,l}(k,l)P_{k-n,j} + b_{l,j+n}(l,j)P_{k,j+n}\,.
\end{equation*}
Using $\{\beta_l\}$, this can be re-written as:
\begin{equation*}
\delta(P_{k,j}) = \beta_{k-n}P_{k-n,j} - \beta_jP_{k,j+n}\,,
\end{equation*}
where, as before, $\beta$ with a negative subscript is defined to be equal to zero.
Consequently, $\beta_l$ is determined uniquely by $\delta$ and we have the following formal sum expression:
\begin{equation*}
\delta(P_{k,j}) = \left[\sum_{l\ge0}\beta_lP_{l,l+n},a\right],
\end{equation*}
which, by continuity of $\delta$, gives \eqref{CovCompact} for all $a\in\mathcal{K}^\infty$.

It remains to prove the growth condition on $\{\beta_j\}$. To this end we study $\delta(a)$ for a special $a\in\mathcal{K}^\infty$,  namely we consider the following:
\begin{equation*}
\delta\left(\sum_{l\ge0}a_{0,l}P_{0,l}\right) = -\sum_{l\ge0}a_{0,l}\beta_lP_{0,l+n}\,.
\end{equation*}
The right-hand side of the above equation must be rapid decay for every RD sequence $\{a_{0,l}\}$.  It therefore follows that $\{\beta_l\}$ is polynomially bounded. Indeed, assuming otherwise, pick an increasing sequence $l_N$ such that 
\begin{equation*}
|\beta_{l_N}|\geq(1+l_N)^N.
\end{equation*}
Then the sequence
\begin{equation*}
a_{0,l} = \left\{
\begin{aligned}
&\frac{1}{\beta_{l_N}} &&\textrm{if }l=l_N\\
&0 &&\textrm{otherwise,}
\end{aligned}\right.
\end{equation*}
is RD  but the product $a_{0,l}\beta_l$ is not.

Finally, because matrix coefficients of $a$ are RD while $\{\beta_l\}$ is at most polynomially increasing, it is clear that the formal sums in \eqref{CovCompact} lead to norm convergent expressions for the values of the derivations $\delta(a)$.
\end{proof}

We are now in a position to classify continuous derivations on $\mathcal{K}^\infty$.

\begin{theo}\label{compact_der_theo}
Let $\delta:\mathcal{K}^\infty\to\mathcal{K}^\infty$ be a continuous derivation, there exists a  sequence $\{\beta_{n,j}\}$ satisfying the following growth condition: for every $r>0$ there is $p>0$ and a constant satisfying
\begin{equation*}
|\beta_{n,l}|\le  \frac{\textrm{const}(1+l)^r}{(1+n)^p}
\end{equation*}
such that $\delta(a)$ is given by the formal sum:  
\begin{equation}\label{delta_compact}
\delta(a)=\left[\sum_{n\ge0}\sum_{j=0}^\infty \beta_{n,j}P_{j,j+n} + \sum_{n<0}\sum_{j=0}^\infty \beta_{n,j}P_{j-n,j},a\right]
\end{equation}
Conversely, if $\{\beta_{n,j}\}$ is a sequence satisfying the above growth condition, then \eqref{delta_compact}
defines a continuous derivation on $\mathcal{K}^\infty$.
\end{theo}
\begin{proof}
Let $\delta:\mathcal{K}^\infty\to\mathcal{K}^\infty$ be a continuous derivation.
The decomposition
\begin{equation*}
\delta(a) = \sum_{n\in\Z}\delta_n(a)
\end{equation*}
follows from Proposition \ref{Fourier_series_norm_conv}.  Since $\delta_n$  are $n$-covariant, it follows from Proposition \ref{n_covariant_der_smooth_compact} that there are sequences $\{\beta_{n,j}\}$ such that $\delta_n(a)$ are given by formulas \eqref{CovCompact}. This establishes equation \eqref{delta_compact} above.

To show the growth conditions of $\{\beta_{n,j}\}$ we apply $\delta_n$ to $P_{0,l}$.  This gives:
\begin{equation*}
\delta_n(P_{0,l}) = \left[\sum_{n\ge0}\beta_{n,j}P_{j,j+n},P_{0,l}\right] = -\beta_{n,l}P_{0,l+n}\,.
\end{equation*}
Using the definition of the $M,N$-norm and equation \eqref{der_k_on_proj}, a straightforward calculation yields
\begin{equation*}
\|P_{0,l}\|_{M,N} = (1+l)^{M+N}\,.
\end{equation*}
Moreover, by Lemma \ref{est_delta_n}, we have for every $M\ge0$, $N\ge0$ and $i\ge0$, there exist $M'\ge0$, $N'\ge0$ and constants $C$ such that
\begin{equation*}
n^i\|\delta_n(a)\|_{M,N}\le C\|a\|_{M',N'}
\end{equation*}
Applying the above inequality to $a=P_{0,l}$ gives
\begin{equation*}
n^i|\beta_{n,l}|\|P_{0,l+n}\|_{M,N}\le C \|P_{0,l}\|_{M',N'}\,.
\end{equation*} 
This in turn yields
\begin{equation*}
|\beta_{n,l}|\le \frac{C(1+l)^{M'+N'}}{n^i(1+l+n)^{M+N}}\le \frac{\textrm{const}(1+l)^r}{(1+n)^p}\quad\textrm{for some }p,r>0\,.
\end{equation*}

Finally, all that remains is to show that if a sequence $\{\beta_{n,j}\}$ satisfies the above growth conditions then the formula
\begin{equation*}
\left[\sum_{n\ge0}\sum_{j=0}^\infty \beta_{n,j}P_{j,j+n} + \sum_{n<0}\sum_{j=0}^\infty \beta_{n,j}P_{j-n,j},a\right]
\end{equation*}
defines a continuous derivation on $\mathcal{K}^\infty$.  We only show the calculation for $n\ge0$ terms as the case $n<0$ is completely analogous. 
Write $a\in\mathcal{K}^\infty$ as $a = \sum_{k,l\ge0} a_{kl}P_{k,l}$
and consider the following calculation:
\begin{equation*}
\begin{aligned}
&\left[\sum_{n\ge0}\sum_{j=0}^\infty \beta_{n,j}P_{j,j+n},a\right] = \sum_{j,k,l,n} \beta_{n,j}a_{k,l}[P_{j,j+n},P_{k,l}]
= \sum_{j,l\ge0}\left(\sum_{n\ge0}\beta_{n,j}a_{j+n,l}\right)P_{j,l}\, - \\
&- \sum_{k,l\ge0}\left(\sum_{l\ge n\ge0}a_{k,l-n}\beta_{n,l-n}\right)P_{k,l}
:= \sum_{j,l\ge0} b_{j,l}P_{j,l} + \sum_{k,l\ge0}c_{k,l}P_{k,l}\,.
\end{aligned}
\end{equation*}
after resummation and computing the commutator.  To finish the proof, we need to show that the coefficients $\{b_{j,l}\}$ and $\{c_{k,l}\}$ are RD sequences.  We first analyze $\{c_{k,l}\}$.  We estimate:
\begin{equation*}
|c_{k,l}|\le \sum_{n=0}^l|a_{k,l-n}||\beta_{n,l-n}|\le \sum_{j=0}^l|a_{k,j}||\beta_{j+l,j}|
\end{equation*}
after re-indexing.  Using the growth condition on $\{\beta_{n,j}\}$ and the fact that $\{a_{k,l}\}$ is a RD sequence we have for every $M>0$:
\begin{equation*}
\begin{aligned}
|c_{k,l}|&\le C\sum_{j=0}^l|a_{k,j}|\frac{(1+j)^N}{(1+j+l)^M} \le \frac{C}{(1+k)^M}\sum_{j=0}^\infty\frac{(1+j)^N}{(1+j)^{N+2}(1+j+l)^M}\\
&\le\frac{\pi^2}{6}\cdot\frac{C}{(1+k)^M(1+l)^M}\,.
\end{aligned}
\end{equation*}
Next we estimate $\{b_{j,l}\}$.  Again, using the growth condition on $\{\beta_{n,j}\}$ and $\{a_{k,l}\}$ being a RD sequence, by taking $M>N$ we have
\begin{equation*}
\begin{aligned}
|b_{jl}| &\le \sum_{n\ge0}|a_{j+n,l}||\beta_{n,j}|\le\sum_{n\ge0}(1+j)^N|a_{j+n,l}| \\
&\le \textrm{const}\sum_{n\ge0}\frac{(1+j)^N}{(1+j+n)^{2M+2}(1+l)^M}\le \frac{\textrm{const}}{(1+j)^M(1+l)^M}\,.
\end{aligned}
\end{equation*}
Thus both $\{b_{j,l}\}$ and $\{c_{k,l}\}$ are RD sequences finishing the proof.
\end{proof}

\section{Classification of Derivations on $\mathcal{T}^\infty$}
In this section we classify derivations on the smooth Toeplitz algebra.  
Namely, we prove that a derivation on $\mathcal{T}^\infty$ is automatically continuous and up to an inner derivation must be a lift of a derivation from $C^\infty(\R/\Z)\cong\mathcal{T}^\infty/\mathcal{K}^\infty$.

Let $\delta:\mathcal{T}^\infty\to\mathcal{T}^\infty$ be a derivation.  We have the following simple observations:

\begin{prop}\label{oper_zero_iff_matrix_coef_zero}
Given an $a$ in $\mathcal{T}$, $a=0$ if and only if for all $i$, $j$, $k$, and $l$, we have
\begin{equation*}
P_{i,j}aP_{k,l}=0\,.
\end{equation*}
\end{prop}
The proof is immediate since the equation means that the matrix coefficients of $a$ are zero.

\begin{prop}\label{der_proj_zero_der_a_zero}
Let $a\in\mathcal{T}^\infty$.  If $\delta(P_{i,j})=0$ for all $i$ and $j$, then $\delta(a)=0$.
\end{prop}

\begin{proof}
Using the Leibniz rule and the assumption that $\delta(P_{i,j})=0$, we have by direct computation that
\begin{equation*}
\delta(P_{i,j}aP_{k,l}) = P_{i,j}\delta(a)P_{k,l}\,.
\end{equation*}
On the other hand, we have:
\begin{equation*}
P_{i,j}aP_{k,l} = a_{j,k}P_{i,l}\,,
\end{equation*}
where $\{a_{j,k}\}$ are the matrix coefficients of $a$ in the basis $\{E_k\}$.
Putting this together with the previous equation yields
\begin{equation*}
 P_{i,j}\delta(a)P_{k,l}=\delta(P_{i,j}aP_{k,l})=0\,.
\end{equation*}
Therefore it follows from Proposition \ref{oper_zero_iff_matrix_coef_zero} that $\delta(a)=0$.
\end{proof}

\begin{prop}
If $\delta_1,\delta_2:\mathcal{T}^\infty\to\mathcal{T}^\infty$ are any two derivations such that
\begin{equation*}
\delta_1(U)=\delta_2(U)\quad\textrm{and}\quad\delta_1(U^*)=\delta_2(U^*)
\end{equation*}
hold, then
\begin{equation*}
\delta_1(a) = \delta_2(a)\,,
\end{equation*}
for all $a\in\mathcal{T}^\infty$.
\end{prop}

\begin{proof}
Consider the map $\delta=\delta_1-\delta_2:\mathcal{T}^\infty\to\mathcal{T}^\infty$.  Then clearly $\delta$ is a derivation on $\mathcal{T}^\infty$ and 
\begin{equation*}
\delta(U)=\delta(U^*)=0\,.
\end{equation*}
Thus for any polynomial, $a$, in $U$ and $U^*$, by the Leibniz rule, we have $\delta(a)=0$.  However, since $P_{i,j}=U^i(1-UU^*)(U^*)^j$ are polynomials in $U$ and $U^*$ and so $\delta(P_{i,j})=0$.  Thus the result follows from Proposition \ref{der_proj_zero_der_a_zero}.
\end{proof}

A key step in classifying derivations on $\mathcal{T}^\infty$ is the following existence theorem.

\begin{theo}\label{b_c_inner_der}
Let $b$ and $c$ be two smooth compact operators so that
\begin{equation}\label{bc_condition}
cU+U^*b=0\,.
\end{equation}
Then, there exists an unique derivation $\delta:\mathcal{T}^\infty\to\mathcal{T}^\infty$ such that
\begin{equation*}
\delta(U)=b\quad\textrm{and}\quad\delta(U^*)=c\,.
\end{equation*}
Moreover, $\delta$ is an inner derivation and its range is contained in $\mathcal{K}^\infty$, that is $\delta(\mathcal{T}^\infty)\subseteq\mathcal{K}^\infty$.
\end{theo}
Notice that the condition \eqref{bc_condition} is necessary because of the following calculation:
\begin{equation*}
0 = \delta(I) = \delta(U^*U) = \delta(U^*)U + U^*\delta(U) = cU+U^*b\,.
\end{equation*}

\begin{proof}
We write $b$ and $c$ in the following way:
\begin{equation*}
b = \sum_{n>0}U^n\beta_n(\K) + \sum_{n\le0}\beta_n(\K)(U^*)^{-n}\quad\textrm{and}\quad c = \sum_{n\ge0}U^n\gamma_n(\K) + \sum_{n<0}\gamma_n(\K)(U^*)^{-n}\,.
\end{equation*}
Using this and the commutation relation from equation \eqref{comm_rel} implies that
\begin{equation*}
\left\{
\begin{aligned}
\gamma_{-1}(\K) &= -\beta_1(\K) \\
\gamma_{n-1}(\K+I) &= -\beta_{n+1}(\K) &&\textrm{for }n>0\\
\gamma_{n-1}(\K) &=-\beta_{n+1}(\K+I) &&\textrm{for }n<0\,.
\end{aligned}\right.
\end{equation*}
Notice that this determines $\gamma$ in terms of $\beta$ except for the $\gamma_n(0)$ term for $n\ge0$ and this also determines $\beta$ in terms of $\gamma$ except for the $\beta_n(0)$ for $n\ge0$.

The goal is to show that there exists an $\alpha\in\mathcal{T}^\infty$ with $\alpha=T(f)+\tilde{\alpha}$ so that $\delta(a) = [\alpha,a]$.  We write $\alpha$ in terms of its Fourier series, that is
\begin{equation*}
\alpha = \sum_{n\ge0}U^n\alpha_n(\K) + \sum_{n<0}\alpha_n(\K)(U^*)^n\,.
\end{equation*}
A direct computation using the desired commutator outcome yields the following formulas for $\delta(U)$ and $\delta(U^*)$ in terms of $\alpha$:
\begin{equation*}
\begin{aligned}
\delta(U) &= \sum_{n\ge0}U^{n+1}(\alpha_n(\K+I)-\alpha_n(\K)) + \sum_{n<0}(\alpha_n(\K)-\alpha_n(\K-I))(U^*)^{-n-1} \\
\delta(U^*) &=\sum_{n>0}U^{n-1}(\alpha_n(\K-I)-\alpha_n(\K)) + \sum_{n\le0}(\alpha_n(\K)-\alpha_n(\K+I))(U^*)^{-n+1}\,,
\end{aligned}
\end{equation*}
where we've used the convention that $\alpha_n(-1)=0$.  By comparing Fourier coefficients, we get 4 sets of relations. The first two sets of relations between $\alpha_n$, $\beta_n$ and $\gamma_n$ are as follows:
\begin{equation*}
\left\{
\begin{aligned}
\alpha_n(k) - \alpha_n(k-1) &=\beta_{n+1}(k) &&\textrm{for }n<0, k\ge0 \\
\alpha_n(k-1) - \alpha_n(k) &= \gamma_{n-1}(k) &&\textrm{for }n>0, k\ge0\,.
\end{aligned}\right.
\end{equation*}
These two recurrence relations can be easily solved for $\alpha_n$, yielding the formulas:
\begin{equation*}
\alpha_n(k) = \left\{
\begin{aligned}
&\sum_{j=0}^k \beta_{n+1}(j) &&\textrm{for }n<0 \\
&-\sum_{j=0}^k\gamma_{n-1}(j) &&\textrm{for }n>0\,.
\end{aligned}\right.
\end{equation*}
There is another set of recurrence relations from making the above comparison, namely:
\begin{equation*}
\left\{
\begin{aligned}
\alpha_n(k) - \alpha_n(k+1) &= \gamma_{n-1}(k) &&\textrm{for }n<0\\
\alpha_n(k+1) - \alpha_n(k) &= \beta_{n+1}(k) &&\textrm{for }n>0\,.
\end{aligned}\right.
\end{equation*}
These follow immediately from our formulas for $\alpha_n(k)$ from above and the relations that relate $\beta_n(k)$ and $\gamma_n(k)$.  However, for the $n=0$ case we get two equivalent equations by virtue of $\gamma_{-1}(k) =-\beta_1(k)$, which are
\begin{equation*}
\begin{aligned}
\alpha_0(k+1)-\alpha_0(k) &= \beta_1(k) \\
\alpha_0(k)-\alpha_0(k+1) &=\gamma_{-1}(k)\,.
\end{aligned}
\end{equation*}
Here the solution will not be unique which corresponds to the non-uniqueness of $\alpha$, at least up to a constant multiple of the identity.  By the normalization, $\alpha_0(0)=0$, solving one of these equivalent relations yields
\begin{equation*}
\alpha_0(k) = \sum_{j=0}^{k-1}\beta_1(j) \quad\textrm{for }k>0\,.
\end{equation*}
By solving these relations for $\alpha$ in terms of $\beta$ and $\gamma$, this demonstrates the existence of such an $\alpha$.  What remains to be shown is that this $\alpha$ actually belongs to $\mathcal{T}^\infty$.  Using our solutions to these relations we can explicitly write the terms to $\{\alpha_n(k)\}_{n\in\Z,k\ge0}$.  In fact we have
\begin{equation*}
\alpha_n(k) = \left\{
\begin{aligned}
&\sum_{j=0}^\infty\beta_1(j) - \sum_{j=k}^\infty\beta_1(j) =f_0 + \tilde{\alpha_0}(k) \\
&\sum_{j=0}^\infty\beta_{n+1}(j) - \sum_{j=k+1}^\infty\beta_{n+1}(j) = f_n + \tilde{\alpha_n}(k) &&\textrm{for }n<0 \\
&-\sum_{j=0}^\infty\gamma_{n-1}(j) + \sum_{j=k+1}^\infty\beta_{n-1}(j) = f_n + \tilde{\alpha_n}(k) &&\textrm{for }n>0\,.
\end{aligned}\right.
\end{equation*}
Now, since $\{\beta_n(k)\}$ and $\{\gamma_n(k)\}$ are RD sequences in both $n$ and $k$, the above series are clearly convergent and moreover, the sequences $\{f_n\}$ and $\{\tilde{\alpha_n}(k)\}$ must be RD sequences as well.  

Consequently, the smooth function $f\in C^\infty(\R/\Z)$ defined with Fourier series
\begin{equation*}
f(x) = \sum_{n\in\Z}f_ne^{2\pi inx}
\end{equation*}
yields a smooth Toeplitz operator
\begin{equation*}
T(f) = \sum_{n\ge0}U^nf_n + \sum_{n<0}f_n(U^*)^{-n}
\end{equation*}
while $\tilde{\alpha}$ given by
\begin{equation*}
\tilde{\alpha} = \sum_{n\ge0}U^n\tilde{\alpha}_n(\K)U^n + \sum_{n<0}\tilde{\alpha}_n(\K)(U^*)^{-n}
\end{equation*}
is a smooth compact operator and we have that $\alpha = T(f) + \tilde{\alpha}$.  This completes the proof.
\end{proof}

Before we are able to state and prove the main decomposition theorem for derivations on the smooth Toeplitz algebra, we need one more proposition.  Similar to the proof of Proposition \ref{smooth_toep_props}, for an $F\in C^\infty(\R/\Z)$, decompose it as $F=F_+ +F_-$, where
\begin{equation*}
F_+(x) = \sum_{n\ge0}F_ne^{2\pi inx}\quad\textrm{and}\quad F_-(x) = \sum_{n<0}F_ne^{2\pi inx}\,.
\end{equation*}

\begin{prop}\label{special_der_delta_F}
The formula
\begin{equation*}
\delta_F(a) = [T(f_+)\K +\K T(f_-),a]
\end{equation*}
defines a continuous derivation on $\mathcal{T}^\infty$ so that following hold for the quotient map:
\begin{equation*}
{q}(\delta_F(U)) = f(x)e^{2\pi ix}\quad\textrm{and}\quad {q}(\delta_F(U^*))=-f(x)e^{-2\pi ix}\,.
\end{equation*}
\end{prop}

\begin{proof}
Since $\delta_F$ is defined though a commutator, using the Leibniz rule, along with the fact that $U$ commutes with $T(f_+)$ and $[\K,U]=U$ we get
\begin{equation*}
\begin{aligned}
\delta_F(U) &= T(f_+)[\K,U] + [\K,U]T(f_-) + \K[T(f_-),U] \\
&= U(T(f_+)+T(f_-)) + \K[T(f_-),U] = UT(f) +\K[T(f_-),U]\,.
\end{aligned}
\end{equation*}
A direct calculation shows that $[T(f_-),U] = P_0T(f_-)U$ and since $\K P_0=0$, we obtain the formula $\delta_F(U)=UT(f)$.  Consequently, we have ${q}(\delta_F(U))=f(x)e^{2\pi ix}$.  A similar argument shows that $\delta_F(U^*)=-T(f)U^*$, and thus the corresponding formula for ${q}(\delta_F(U^*))$ follows.

Clearly $\delta_F$ is a continuous derivation on $\mathcal{K}^\infty$, so we only need to study its action on a $T(f)$ for $f\in C^\infty(\R/\Z)$.  Using the decomposition $f=f_+ + f_-$ we can compute $\delta_F$ on $T(f)$ as follows:
\begin{equation*}
\begin{aligned}
\delta_F(T(f_-)) &= [T(f_+)\K + \K T(f_-),T(f_-)]\\
&= T(f_+)[\K,T(f_-)] + [T(f_+),T(f_-)]\K + [\K,T(f_-)]T(f_-)\,,
\end{aligned} 
\end{equation*}
since $T(f_-)$ and $T(f_-)$ commute with each other.  The first and the third term in the above equation can be combined and we arrive at
\begin{equation*}
\delta_F(T(f_-)) = T(f)T\left(\frac{1}{2\pi i}\frac{d}{dx}f_-\right) + [T(f_+),T(f_-)]\K\,.
\end{equation*}
It follows from Proposition \ref{smooth_toep_props}, that $[T(f_+),T(f_-)]$ is in $\mathcal{K}^\infty$ and hence the second term in the above equation is also in $\mathcal{K}^\infty$.  A similar consideration applies to $\delta_F(T(f_+))$.  By the remarks after Proposition \ref{smooth_toep_props}, it now follows that $\delta_F$ is a continuous derivation on $\mathcal{T}^\infty$.
\end{proof}

Notice that for the derivation $\delta_F$ defined the above proposition, given an $a\in\mathcal{T}^\infty$, we have for the quotient map
\begin{equation*}
{q}(\delta_F(a)) = F(x)\frac{1}{2\pi i}\frac{d}{dx}{q}(a)(x)\,.
\end{equation*}
Hence, $\delta_F$ is a lift of the derivation
\begin{equation*}
F(x)\frac{1}{2\pi i}\frac{d}{dx}
\end{equation*}
on $C^\infty(\R/\Z)$.

\begin{theo}
If $\delta:\mathcal{T}^\infty\to\mathcal{T}^\infty$ is a derivation, then $\delta$ is continuous and there exists a unique function $F\in C^\infty(\R/\Z)$ such that 
\begin{equation*}
\delta = \delta_F + \tilde{\delta}\,,
\end{equation*}
where $\delta_F$ was the derivation defined in Proposition \ref{special_der_delta_F} and $\tilde{\delta}:\mathcal{T}^\infty\to\mathcal{K}^\infty$ is an inner derivation.
\end{theo}

\begin{proof}
Define the function $F\in C^\infty(\R/\Z)$ through the quotient map equation
\begin{equation*}
{q}(\delta(U))(x) = F(x)e^{2\pi ix}\,.
\end{equation*}
Since
\begin{equation*}
0=\delta(U^*U) = U^*\delta(U) + \delta(U^*)U\,,
\end{equation*}
by applying the quotient map ${q}$ to above equation, it follows that
\begin{equation*}
{q}(\delta(U^*)) = -F(x)e^{-2\pi ix}\,.
\end{equation*}
Define the map $\tilde{\delta}:\mathcal{T}^\infty\to\mathcal{T}^\infty$ by
\begin{equation*}
\tilde{\delta} = \delta - \delta_F\,.
\end{equation*}
It's clear that $b:=\tilde{\delta}(U)$ and $c:=\tilde{\delta}(U^*)$ are both in $\mathcal{K}^\infty$ by the definition of $F$.  Since
\begin{equation*}
0=\tilde{\delta}(U^*U) = U^*\tilde{\delta}(U) + \tilde{\delta}(U^*)U\,,
\end{equation*}
we see that $b$ and $c$ satisfy the assumptions of Theorem \ref{b_c_inner_der}.   Consequently, $\tilde{\delta}$ is an inner derivation with range contained in $\mathcal{K}^\infty$ and in particular, it is continuous.  Since $\delta_F$ is continuous by Proposition \ref{special_der_delta_F}, $\delta$ is automatically continuous.  The uniqueness of $F$ follows from the fact that $\delta_F$ is inner if and only if $F\equiv0$ since the quotient map ${q}$ acting on any inner derivation is always equal to $0$.
\end{proof}

\section{Automorphisms}
In this section we discuss the group of holomorphic automorphisms of the classical disk and its natural action on the noncommutative disk $\mathcal{T}$. The construction is similar to \cite{KL} and we additionally prove here that such automorphisms preserve the smooth subalgebra $\mathcal{T}^\infty$ of the Toeplitz algebra $\mathcal{T}$. 

Consider the Lie group:
\begin{equation*}
\textrm{SU}(1,1)=\left\{g=\begin{bmatrix}
\alpha & \beta\\
\bar\beta & \bar\alpha
\end{bmatrix}, |\alpha|^2-|\beta|^2=1\right\}.
\end{equation*}
Given $g\in\textrm{SU}(1,1)$, the M\"{o}bius transformation (denoted by the same letter)
\begin{equation}\label{g_formula}
g(z)=\frac{\alpha z+\beta}{\bar\beta z+\bar\alpha}
\end{equation}
is a biholomorphism of the open unit disk $\{z\in\C:|z|<1\}$ and every biholomorphisms of the disk is of this form. Moreover, each transformation preserves the boundary of the disk defining a smooth diffeomorphism of the circle $S^1=\{z\in\C:|z|=1\}$.

It was observed in \cite{KL} that the group $\textrm{SU}(1,1)$ acts on $\mathcal{T}$ by a formula like \eqref{g_formula} applied to a generator of $\mathcal{T}$. We will repeat some of the arguments from \cite{KL}, however applied to a different generator of $\mathcal{T}$, namely $U$. 
\begin{prop}\label{invert_prop}
With the above notation, $\bar\beta U+\bar\alpha$ is invertible in $\mathcal{T}$ for every $g\in\textrm{SU}(1,1)$. 
Moreover, we have 
$$(\bar\beta U+\bar\alpha)^{-1}\in \mathcal{T}^\infty.$$
\end{prop}
\begin{proof} We write $\bar\beta U+\bar\alpha=\bar\alpha\left(I+\frac{\bar\beta}{\bar\alpha}U\right)$ and notice that
\begin{equation*}
\left\|\frac{\bar\beta}{\bar\alpha}U\right\|^2=\frac{|\beta|^2}{|\alpha|^2}=\frac{|\beta|^2}{1+|\beta|^2}<1,
\end{equation*}
so that $\bar\beta U+\bar\alpha$ is invertible. Since 
$$\bar\beta U+\bar\alpha=T\left(\bar\beta z+\bar\alpha\right)$$
is a smooth Toeplitz operator, its inverse must also be smooth by Proposition \ref{ToepStableProp}.
\end{proof}

\begin{prop}
For every $g\in\textrm{SU}(1,1)$ let 
$$W_g:=(\alpha U+\beta)(\bar\beta U+\bar\alpha)^{-1}.$$ 
Then we have $W_g\in\mathcal{T}^\infty$ and $W_g^*W_g=I$.
\end{prop}
\begin{proof} The first part $W_g\in\mathcal{T}^\infty$ follows from the previous while the second part $W_g^*W_g=I$ is a result of a straightforward calculation:
\begin{equation*}
I-W_g^*W_g=(\beta U^*+\alpha)^{-1}(I-U^*U)(\bar\beta U+\bar\alpha)^{-1}=0
\end{equation*}
\end{proof}

The correspondence $U\mapsto W_g$ extends to an automorphism $\rho_g$ of the C$^*$-algebra $\mathcal{T}$ such that:
\begin{equation*}
\rho_g(U)=W_g
\end{equation*}
because of Coburn's universality theorem for the Toeplitz algebra. Consequently, we obtain an action $g\mapsto \rho_g$ of $\textrm{SU}(1,1)$  on $\mathcal{T}$ by automorphisms, and the action is easily seen to be continuous.

To proceed further we want to construct unitary operators $\mathcal{U}_g$ implementing the automorphisms $\rho_g$:
\begin{equation}\label{implement}
 \rho_g(a)=\mathcal{U}_g a \mathcal{U}_g^{-1}.
\end{equation}
To this end we need to first study the kernel of $W_g^*$.
\begin{prop}
The kernel  of $W_g^*$ is one dimensional and is spanned by
\begin{equation*}
F_0=\sum_{k=0}^\infty\frac{(-\bar\beta)^k}{(\bar\alpha)^{k+1}}E_k\,.
\end{equation*}
Additionally, we have $F_0\in S_\mathcal{E}$, where $S_\mathcal{E}$ is the RD subspace defined in \eqref{S_def} and $\|F_0\|=1$.
\end{prop}
\begin{proof} Notice that the equation $W_g^*x=0$ is equivalent to:
\begin{equation*}
(\bar\alpha U^*+\bar\beta)\sum_{k=0}^\infty x_kE_k=0,
\end{equation*}
which is a one-step recurrence equation on the coefficients $\{x_k\}$ and hence has one dimensional solution space.
It is a straightforward calculation to verify that $F_0$ above is a solution and that $\|F_0\|=1$. Since the coefficients of $F_0$ decrease exponentially, we have $F_0\in S_\mathcal{E}$.
\end{proof}
Notice for future use that we have the following formula for $F_0$ in terms of $E_0$:
\begin{equation*}
F_0=(\bar\alpha+\bar\beta U)^{-1}E_0\,.
\end{equation*}
If we define 
$$F_k:=W_g^kF_0$$ 
then, since $W_g$ is an isometry and $\|F_0\|=1$, it is easy to verify that $\{F_k\}$ is an orthonormal basis. Consequently, we can define an unitary operator $\mathcal{U}_g$ by:
\begin{equation*}
\mathcal{U}_gE_k:=F_k,\ \  k\geq 0\,.
\end{equation*}
Clearly we have $\mathcal{U}_g U \mathcal{U}_g^{-1}=W_g$
and so the unitary operators $\mathcal{U}_g$ implement the automorphisms $\rho_g$, that is \eqref{implement} holds. We are now ready to formulate the main result of this section.
\begin{theo}
With the above notation, for every $g\in\textrm{SU}(1,1)$ we have:
\begin{equation*}
\rho_g(\mathcal{T}^\infty)\subseteq \mathcal{T}^\infty\,.
\end{equation*}
\end{theo}
\begin{proof} First notice that for any $f\in C(S^1)$ we have:
\begin{equation*}
\rho_g(T(f))=T(f\circ g)\,.
\end{equation*}
This is certainly true for polynomials in $U$ and $U^*$ and so by continuity it is true for any $f\in C(S^1)$. Consequently we need to prove that for every $g\in\textrm{SU}(1,1)$ we have:
\begin{equation*}
\rho_g(\mathcal{K}^\infty)\subseteq \mathcal{K}^\infty\,.
\end{equation*}
Using Proposition \ref{Comp_Inv} it then enough to prove that $\mathcal{U}_g: S_\mathcal{E}\to S_\mathcal{E}$ is a continuous map of Fr\'{e}chet spaces, where $S_\mathcal{E}\subset H$ is the dense subspace of $H$ of linear combinations of the basis elements $\{E_k\}$ with rapid decay coefficients.

It is easier to work with the following operator $\mathcal{V}_g$ instead of $\mathcal{U}_g$:
\begin{equation*}
\mathcal{V}_g:=(\bar\alpha+\bar\beta U)\,\mathcal{U}_g\,.
\end{equation*}
Below, we need the following observation. 
\begin{lem} \label{TfS_lemma}
For every $f\in C^\infty(S^1)$ and every $x\in S_\mathcal{E}$ we have:
\begin{equation*}
\|T(f)x||_N\leq \|f\|_{C^N}\|x||_N\,.
\end{equation*}
\end{lem}
\begin{proof} The case of $N=0$ is just \eqref{Tfnorm}. Inducting on $N$ we have:
\begin{equation*}
\begin{aligned}
&\|T(f)x||_{N+1}=\|(I+\K)T(f)x||_{N}\leq \|T(f)(I+\K)x\|_N.+ \|\delta_\K(T(f))x\|_N\leq \|f\|_{C^N}\|x||_{N+1}+\\
&+ \left\|\frac{1}{2\pi i}\frac{d}{dx}\,f\right\|_{C^N}\|x||_{N}
\leq \left(\|f\|_{C^N}+\left\|\frac{1}{2\pi i}\frac{d}{dx}\,f\right\|_{C^N}\right)\|x||_{N+1}=\|f\|_{C^{N+1}}\|x||_{N+1}.
\end{aligned}
\end{equation*}
\end{proof}
It now follows from Proposition \ref{invert_prop} and the above Lemma that $\mathcal{U}_g: S_\mathcal{E}\to S_\mathcal{E}$ is  continuous iff $\mathcal{V}_g: S_\mathcal{E}\to S_\mathcal{E}$ is continuous. We prove by induction on $N$ that we have an inequality:
\begin{equation*}
\|\mathcal{V}_gx\|_N\leq\textrm{const}\|x\|_N,
\end{equation*}
for some constant depending on $N$ and $g$ that we will not track.

The case of $N=0$ is just the boundedness of $\mathcal{V}_g$. For the inductive step we estimate:
\begin{equation}\label{vgest}
\|\mathcal{V}_gx||_{N+1}=\|(I+\K)\mathcal{V}_gx||_{N}\leq \|\mathcal{V}_g(I+\K)x\|_N.+ \|\delta_\K(\mathcal{V}_g)x\|_N.
\end{equation}
Thus, we need to understand the derivation $\delta_\K(\mathcal{V}_g)$. In fact, we have the following formula in terms of $W_g$:
\begin{equation}\label{delta_k_v}
\delta_\K(\mathcal{V}_g)=(\delta_\K(W_g)W_g^*-I)\mathcal{V}_g\K\,.
\end{equation}
To prove it, we apply both sides to basis elements $E_k$. First, notice that we have:
\begin{equation*}
\mathcal{V}_g E_k=(\bar\alpha+\bar\beta U)\,\mathcal{U}_g E_k=(\bar\alpha+\bar\beta U)F_k=(\bar\alpha+\bar\beta U)W_g^kF_0=
W_g^k(\bar\alpha+\bar\beta U)F_0=W_g^kE_0\,.
\end{equation*}
Consequently, since $\K E_k=k E_k$ and $W_g^*W_g=I$, we obtain:
\begin{equation*}
\begin{aligned}
&\delta_\K(\mathcal{V}_g)E_k=\K \mathcal{V}_g E_k-\mathcal{V}_g\K E_k= \K W_g^kE_0 - \mathcal{V}_g\K E_k= \delta_\K(W_g^k)E_0 - \mathcal{V}_g\K E_k  =\\
&=k\delta_\K(W_g)W_g^{k-1}E_0 - \mathcal{V}_g\K E_k=k\delta_\K(W_g)W_g^*W_g^{k}E_0 - \mathcal{V}_g\K E_k=
(\delta_\K(W_g)W_g^*-I)\mathcal{V}_g\K E_k\,.
\end{aligned}
\end{equation*}
Next, since $\delta_\K(U)=U$, we obtain from the definition of $W_g$:
\begin{equation*}
\delta_\K(W_g)=U(\bar\alpha+\bar\beta U)^{-2}\,.
\end{equation*}
Thus, we can rephrase  formula \eqref{delta_k_v} in terms of Toeplitz operators as:
\begin{equation*}
\delta_\K(\mathcal{V}_g)=\left(T(z(\bar\alpha+\bar\beta z)^{-2})T((\bar\alpha \bar z+\bar\beta)(\beta \bar z+\alpha)^{-1})-I\right)\mathcal{V}_g\K\,.
\end{equation*}
Returning to the inductive step \eqref{vgest}, we obtain using Lemma \ref{TfS_lemma} and the inductive hypothesis $\|\mathcal{V}_gx\|_N\leq\textrm{const}\|x\|_N$:
\begin{equation*}
\begin{aligned}
&\|\mathcal{V}_gx||_{N+1}\leq \|\mathcal{V}_g(I+\K)x\|_N + \|\delta_\K(\mathcal{V}_g)x\|_N\leq \textrm{const}\|(I+\K)x\|_N+\textrm{const}\|\mathcal{V}_g\K x\|_N\leq\\
&\leq \textrm{const}\|(I+\K)x\|_N=\textrm{const}\|x\|_{N+1}\,,
\end{aligned}
\end{equation*}
finishing the proof.
\end{proof} 

\section{Remarks on K-Theory, K-Homology and Cyclic Cohomology}
In this section we give a discussion of the $K$-Theory of the smooth subalgebras $C^\infty(\R / \Z )$, $\mathcal{K}^\infty$, and $\mathcal{T}^\infty$. We also give a discussion of the $K$-Homology of their $*$-completions $C(\R / \Z)$, $\mathcal{K}$, and $\mathcal{T}$. First, we investigate the $6$-term exact sequence in $K$-Theory induced by the short exact sequence of the Toeplitz algebra, exhibiting explicit generators of the $K$-Theory, and computing each induced map in $K$-Theory on these generators.  We then compute the $K$-homology of these algebras, finding explicit Fredholm modules whose classes generate the $K$-homology. We also discuss spectral triples and cyclic cohomology for the above mentioned algebras. The geometrical aspects of those otherwise topological considerations lie in constructions of generators of the homology/cohomology groups. Most of the material in this section is well-known, perhaps with the exception of constructions of some Fredholm modules and spectral triples. The section is intended to give a fuller picture of the noncommutative geometry of the quantum disk.

\subsection{Review of $K$-Theory} Note that, by earlier discussions, $\mathcal{K}^\infty$ and $\mathcal{T}^\infty$ are subalgebras of $\mathcal{K}$ and $\mathcal{T}$, respectively, which are closed under the holomorphic calculus. Similarly, this fact is well known for the subalgebra $C^\infty(\R / \Z)$ of $C(\R / \Z)$. Hence, each inclusion induces an isomorphism in $K$-Theory and we have the following proposition summarizing the $K$-Theory groups and generators. 
\begin{prop}
The table below summarizes the $K$-Theory of the smooth subalgebras $C^\infty(\R / \Z ), \mathcal{K}^\infty$, and $\mathcal{T}^\infty$.
\begin{center}
\begin{tabular}{ |c|c|c| } 
 \hline
  & Group & Generator \\ 
 \hline
 $K_0(C^\infty(\R / \Z))$ & $\Z$ & $[1]_0$ \\ 
 \hline
 $K_0(\mathcal{K}^\infty)$ & $\Z$ & $[P_{0,0}]_0$ \\ 
 \hline
  $K_0(\mathcal{T}^\infty)$ & $\Z$ & $[I]_0$ \\ 
 \hline
 $K_1(C^\infty(\R / \Z))$ & $\Z$ & $[z]_1$ \\ 
 \hline
 $K_1(\mathcal{K}^\infty)$ & $0$ &  NA\\ 
 \hline
  $K_1(\mathcal{T}^\infty)$ & $0$ &  NA\\ 
  \hline
\end{tabular}
\end{center}
\end{prop}
Here we identify $\R / \Z$ with the circle $S^1=\{\zeta\in\C:|\zeta|=1\}$ and the function $z\in C^\infty(\R / \Z)$ in the above table is the identity function $z(\zeta)=\zeta\in\C$.

Note that we have the following short exact sequence of subalgebras 
\begin{equation*}
\begin{tikzcd}
0 \arrow{r} & \mathcal{K}^\infty \arrow{r}{\iota} & \mathcal{T}^\infty \arrow{r}{q} & C^\infty(\R / \Z) \arrow{r} & 0,  \\
\end{tikzcd}
\end{equation*}
which induces a $6$-term exact sequence in $K$-Theory 
\begin{equation*}
\begin{tikzcd}
 K_0(\mathcal{K}^\infty) \arrow{r}{K_0(\iota)}  & K_0(\mathcal{T}^\infty)  \arrow{r}{K_0(q)} & K_0(C^\infty(\R / \Z))  \arrow{d}{\exp}  \\
 K_1(C^\infty(\R / \Z))\arrow{u}{\textrm{ind}} & K_1(\mathcal{T}^\infty) \arrow{l}{K_1(q)} & K_1(\mathcal{K}^\infty). \arrow{l}{K_1(\iota)}
\end{tikzcd}
\end{equation*}
Due to the vanishing of $K_1(\mathcal{K}^\infty)$ and $K_1(\mathcal{T}^\infty)$, the induced maps $K_1(q),K_1(\iota)$, $\exp$ are all trivial homomorphisms. It is simple to compute 
$$K_0(q)[I]_0 = [q(I)]_0 = [1]_0.$$ 
By exactness, $K_0(\iota)$ is trivial. The connecting homomorphism, the index map:
\begin{equation*}
\textrm{ind}: K_1(C^\infty(\R / \Z )) \to K_0(\mathcal{K}^\infty)
\end{equation*}
can be computed as follows. Note that the generator of $K_1(C^\infty(\R / \Z))$, the identity function, lifts to the element $U$ in $\mathcal{T}^\infty$. Using Proposition $9.2.4$ in \cite{RLL}, we have that 
$$\textrm{ind}([z]_1) = [I - U^*U]_0 - [I - UU^*]_0 = -[P_{0,0}]_0.$$ 
This completes our computation of the induced maps in $K$-Theory.

\subsection{$K$-Homology Groups } Recall that the definition of an odd Fredholm module over a $*$-algebra $A$ is a triple $(H, \rho , F)$ where $H$ is a Hilbert space, $\rho: A \to B(H)$ is a $*$-representation, and $F \in B(H)$ satisfies 
$$F^*- F,\ I - F^2 \in \mathcal{K}(H),\textrm{ and }[F, \rho(a)] \in \mathcal{K}(H)$$ 
for all $a \in A$. Similarly, an even Fredholm module is the above information, together with a $\Z_2$-grading $\Gamma$ of $H$, such that 
$$\Gamma^* = \Gamma,\ \Gamma^2 = I,\ \Gamma F = - F \Gamma,\ \textrm{and  }\Gamma \rho(a) = \rho(a) \Gamma$$ 
for all $a \in A$.  Classes of odd Fredholm modules over $A$ generate $K^1(A) : = KK^1(A, \mathbb{C})$, while classes of even Fredholm modules over $A$ generate $K^0(A) : = KK^0(A, \mathbb{C})$. The unbounded analogues of Fredholm modules are called spectral triples, which are important tools in noncommutative geometry and give rise to a noncommutative generalization of a Riemannian spin manifold. In what follows, we compute the $K$-homology of each C$^*$-algebra in the short exact sequence of the Toeplitz algebra, exhibiting generating Fredholm modules in each case. We also give several remarks regarding spectral triples as unbounded versions of specific Fredholm modules. 

All algebras being studied fall into the Bootstrap class \cite{BB}, and hence the Universal Coefficient Theorem of Rosenberg and Schochet \cite{RS UCT} states that we have an exact sequence 
\begin{equation*}
    \begin{tikzcd}
    0 \arrow{r} & \Ext^1_\Z(K_1(A), \Z) \arrow{r} & K^0(A) \arrow{r} & \Hom(K_0(A),\Z) \arrow{r} & 0,
    \end{tikzcd}
\end{equation*}
where $A$ is any one of $C(\R / \Z), \mathcal{T}, \mathcal{K}$. Similarly, we have an exact sequence 
\begin{equation*}
    \begin{tikzcd}
    0 \arrow{r} & \Ext^1_\Z(K_0(A), \Z) \arrow{r} & K^1(A) \arrow{r} & \Hom(K_1(A),\Z) \arrow{r} & 0.
    \end{tikzcd}
\end{equation*}
Since $\Ext^1_\Z(\Z, \Z) \cong 0, \textrm{ and } \Ext^1_\Z(0, \Z) \cong 0$, it follows immediately that we have the following proposition summarizing the $K$-homology groups. 
\begin{prop}
The $K$-homology groups of $C(\R / \Z), \mathcal{K}$, and  $\mathcal{T}$ are summarized in the table below. 
\begin{center}
\begin{tabular}{ |c|c| } 
 \hline
 $K^0(C(\R / \Z))$ & $\Z$ \\ 
 \hline
 $K^0(\mathcal{K})$ & $\Z$ \\ 
 \hline
  $K^0(\mathcal{T})$ & $\Z$  \\ 
 \hline
 $K^1(C(\R / \Z))$ & $\Z$ \\ 
 \hline
 $K^1(\mathcal{K})$ & $0$ \\ 
 \hline
  $K^1(\mathcal{T})$ & $0$ \\ 
  \hline
\end{tabular}
\end{center}
\end{prop}

\subsection{Fredholm Modules and Spectral Triples} 
We now construct explicit Fredholm modules for each nontrivial $K$-homology group above, and verify via Connes' index theorem that each Fredholm module is a generator for the corresponding $K$-homology group. 
\subsubsection{Fredholm Modules over $C(\R / \Z)$}
Consider the finite dimensional $*$-representation $\rho_0: C(\R / \Z) \to B(\{0\} \oplus \C)$ given by 
$$\rho_0(f) = 0 \oplus f(1).$$ 
Since the Hilbert space $\{0\} \oplus \C$ is finite dimensional, all necessary conditions for $F$ in the definition of Fredholm module are clearly satisfied by 
$$F_0 = \begin{bmatrix}0 & 0 \\ 0 & 0 \end{bmatrix}.$$ 
Equipped with the grading operator 
$$\Gamma = \begin{bmatrix}0 & 0 \\ 0 & -1 \end{bmatrix},$$ 
it follows that $(\{0\} \oplus \C, \rho_0, F_0, \Gamma)$ as above is an even Fredholm module over $C(\R /\Z)$. The following proposition verifies that the class of this Fredholm module generates the even $K$-homology of $C(\R / \Z)$. 
\begin{prop}\label{ev FM circle}
The following triple 
\begin{equation*}
\left(\{0\} \oplus \C, \rho_0, F_0\right)
\end{equation*}
together with grading operator $\Gamma$ defines an even Fredholm module over $C(\R /\Z)$ whose class generates the  even $K$-homology $K^0(C(\R / \Z)) \cong \mathbb{Z}.$
\end{prop}
\begin{proof}
That $(\{0\} \oplus \C, \rho_0, F_0, \Gamma)$ is an even Fredholm module is verified in the preceding paragraph. To show its class generates $K^0(C(\R / \Z))$, it suffices to show that the index pairing with the generator of $K$-theory $[1]_0$ is $1$ \cite{Connes}. By the Connes' index theorem \cite{Connes}, this pairing is computed as the index of the $0$ operator $ \textrm{Ran} (1) = \C \to \{0\} $. The index of such an operator between finite dimensional Hilbert spaces is simply $\dim \C - \dim \{0\} = 1$. This completes the proof. 
\end{proof}

We now construct an odd Fredholm module over $C(\R / \Z)$ whose class generates the odd $K$-homology $K^1(C(\R / \Z)) \cong \Z$. Define a $*$-representation $\rho_1: C(\R/ \Z) \to \ell^2(\Z)$ by 
$$\rho_1(f) = f(V),$$ 
where $V$ denotes the bilateral shift operator 
$$VE_k = E_{k+1}.$$ 
Define Fredholm operator $F_1: \ell^2(\Z) \to \ell^2(\Z)$ by 
$$
F_1E_k = 
\left\{
\begin{aligned}
 & E_k \qquad \textrm{if } k\geq 0 \\
 -& E_k \qquad \textrm{if } k<0\,.
\end{aligned}\right.
$$
We have the following proposition.
\begin{prop}
The triple $(\ell^2(\Z), \rho_1 , F_1)$ defines an odd Fredholm module over $C(\R / \Z)$ whose class generates the odd $K$-homology $K^1(C(\R / \Z))$. 
\end{prop}
\begin{proof}
$F_1$ is clearly self adjoint, and its square is precisely the identity.  Moreover, $[F_1, \rho(f)]$ is easily seen to be of finite rank at most $n$ for any polynomial $f$ of degree $n$. Hence, $[F_1, \rho(f)] \in \mathcal{K}(\ell^2(\Z))$ for any $f \in C(\R / \Z)$. To see that this Fredholm module generates the odd $K$-homology, it suffices show that its pairing with the generator $[z]_1$ of $K_1(C(\R / \Z))$ is $\pm 1$. By the Connes' index theorem \cite{Connes}, this is computed as the index of the operator
$$P_{\geq 0}\rho_1(z)P_{\geq 0} = P_{\geq 0} U P_{\geq 0},$$ 
where 
$$P_{\geq 0} = \frac{1}{2}(I + F_1)$$ is the spectral projection of $F_1$. This is easily seen to be equal to
$$\textrm{dim ker}(P_{\geq 0} U P_{\geq 0}) - \textrm{dim ker}(P_{\geq 0} U^* P_{\geq 0}) = 0 -1 = -1$$ 
completing the proof.
\end{proof}

\subsubsection{Spectral Triples over $C(\R / \Z)$}
If in the above triple, the bounded operator $F$ is replaced with unbounded label operator $\mathbb{K}E_k = kE_k$, then we obtain an odd spectral triple $(C^\infty(\R / \Z) , \ell^2(\Z), \mathbb{K})$ over $C(\R / \Z)$. Indeed, this follows from the simple observation that $[\K, \rho(f)]$ is bounded for all $f$ in the dense subalgebra of smooth functions on $\R / \Z$.  

Also note that when the representation is given by $\rho:A \to B(H_{\textrm{fin}})$ for a finite dimensional Hilbert space $H_{\textrm{fin}}$, the notions of spectral triple and Fredholm modules coincide. Hence, the above even Fredholm module also defines a spectral triple on $C(\R / \Z)$. 

\subsubsection{Fredholm Modules over $\mathcal{K}$} 
Since the odd $K$-homology of $\mathcal{K}$ is trivial, we see there are no topologically nontrivial odd Fredholm modules over $\mathcal{K}$. For the discussion of even Fredholm modules, since $\mathcal{K}$ is nonunital it will be more convenient to consider the unitization $\mathcal{K}^+$. The even $K$-theory of $\mathcal{K}^+$ is given by 
$$K_0(\mathcal{K}^+) = \Z \oplus \Z$$ 
with generators $[I]_0$ and $[P_{0,0}]_0$. In what follows we construct a Fredholm module that pairs nontrivially with the generator $[P_{0,0}]_0$, but trivially with the generator of the unitization part $[I]_0$. First, recall that here we study the algebra of compact operators $\mathcal{K}(H)$ for some infinite, separable Hilbert space $H$, equipped with an orthonormal basis $\{E_k\}_{k \geq 0}$. Consider two operators, $U_1$ and $U_2$, acting on $H$ and satisfying
$$
U^*_i U_j = \delta_{i,j}I, \quad U_1U_1^* + U_2U_2^* = I.
$$
We remark that the universal C$^*$-algebra generated by such operators is well known and is of independent interest \cite{EL}. Define $H_{ev} = H$, and $H_{odd} = H \oplus H$. Let $\rho_{ev}: \mathcal{K}^+(H) \to B(H_{ev})$ be the canonical representation, so that 
$$\rho_{ev}(P_{i,j}) E_{k} = \chi_{j,k}E_i.$$ 
Let $\rho_{odd}: \mathcal{K}^+(H) \to B(H_{odd})$ be given by 
$$\rho_{odd} = \rho_{ev} \oplus \rho_{ev}.$$  
Define an operator $G^*: H_{ev} \to H_{odd}$ by
$$
G^*(x) = \begin{bmatrix} U_1^*x \\ U_2^*x
\end{bmatrix},
$$
with adjoint $G: H_{odd} \to H_{ev}$ given by
$$
G \begin{bmatrix} x \\ y
\end{bmatrix} = U_1(x) + U_2(y).
$$
We have the following proposition.
\begin{prop}
The triple
$$
\left(H_{ev} \oplus H_{odd}, \rho_{ev} \oplus \rho_{odd}, F_2 = \begin{bmatrix} 0 & G \\G^* & 0 \end{bmatrix} \right)
$$ 
together with grading operator 
$$
\begin{bmatrix} 1 & 0 \\ 0 & -1 \end{bmatrix}
$$
defines an even Fredholm module over $\mathcal{K}^+$ whose class in $K^0(\mathcal{K^+}) \cong K^0(\mathcal{K}) \oplus \Z$ generates  $K^0(\mathcal{K})$. 
\end{prop}
\begin{proof}
$F_2$ is clearly self adjoint, and its square is easily computed to be the identity:
$$
\begin{bmatrix} 0 & G \\G^* & 0 \end{bmatrix}\begin{bmatrix} 0 & G \\G^* & 0 \end{bmatrix} \begin{bmatrix} x \\ y \\ z \end{bmatrix} = \begin{bmatrix} G \begin{pmatrix} U_1^* x \\ U_2^* x \end{pmatrix} \\ G^*(U_1(y) + U_2(z)) \end{bmatrix} = \begin{bmatrix} U_1U_1^*x + U_2U_2^*x \\ U_1^*U_1(y) + U^*_1U_2(z) \\ U_2^*U_1(y) + U_2^*U_1(z) \end{bmatrix} = \begin{bmatrix} x \\ y \\ z  \end{bmatrix}.
$$
Since compacts are an ideal of $B(H_{ev})$, $B(H_{odd})$, and $\rho_{ev}$, $\rho_{odd}$ have ranges of the form $ \textrm{compact} + I$, it follows that the commutator $[F_2, (\rho_{ev} \oplus \rho_{odd})(a)]$ is compact for any $a \in \mathcal{K}^+$. Hence, the above indeed defines an even Fredholm module over $\mathcal{K}^+$. It only remains to check that the class of this Fredholm module pairs nontrivially with the nonunital part $K^0(\mathcal{K})$ of $K^0(\mathcal{K}^+)$, and trivial with $[I]_0$. Again by Connes' index theorem, this is reduced to computing two indices. We first check that the pairing with the unital part $[I]_0$ is trivial. Indeed, 
$$
\langle [I]_0 , (H_{ev} \oplus H_{odd}, \rho_{ev} \oplus \rho_{odd}, F_2) \rangle = \textrm{Index}(\rho_{ev}(I) G \rho_{odd}(I)).
$$
It is clear that the above index is $0$. Indeed, for any $x \in H$ we have 
$$G \begin{bmatrix} U_1^*x \\ U_2^*x \end{bmatrix} = x,$$ 
so $G$ has no cokernel. To see $G$ has no kernel, applying $U^*_1$ on the left to the equation
$$G \begin{bmatrix} x \\ y \end{bmatrix} = 0$$ 
shows that $x = 0$, while applying $U_2^*$ on the left to the above relation shows $y = 0$.  This shows that 
$$
\langle [I]_0 , (H_{ev} \oplus H_{odd}, \rho_{ev} \oplus \rho_{odd}, F_2) \rangle = 0.
$$

We now compute the pairing $\langle [P_{0,0}]_0 , (H_{ev} \oplus H_{odd}, \rho_{ev} \oplus \rho_{odd}, F_2) \rangle$. Again, by the Connes' index theorem we have 
$$
\langle [P_{0,0}]_0 , (H_{ev} \oplus H_{odd}, \rho_{ev} \oplus \rho_{odd}, F_2) \rangle = \textrm{Index}(\rho_{ev}(P_{0,0}) G \rho_{odd}(P_{0,0})), 
$$
where we consider above the restricted operator
$$\rho_{ev}(P_{0,0}) G \rho_{odd}(P_{0,0}): \textrm{Ran}( \rho_{odd}(P_{0,0})) \to \textrm{Ran}(\rho_{ev}(P_{0,0}))\,.$$ 
Since  
$$\textrm{Ran}( \rho_{odd}(P_{0,0})) = \textrm{span} \{ E_0 \oplus 0, 0 \oplus E_0 \}\textrm{ and }\textrm{Ran}(\rho_{ev}(P_{0,0}) ) = \textrm{span}\{ E_0\}$$
are finite dimensional, it follows that 
$$
\langle [P_{0,0}]_0 , (H_{ev} \oplus H_{odd}, \rho_{ev} \oplus \rho_{odd}, F_2) \rangle = 2 - 1 = 1,
$$
and hence the above Fredholm module generates the part of $K^0(\mathcal{K}^+)$ corresponding to $K^0(\mathcal{K})$. This completes the proof. 
\end{proof}

\subsubsection{Spectral triples over $\mathcal{K}$}
 In \cite{KMP1}, the authors of this paper claimed to construct an even spectral triple over $\mathcal{T}$. Due to a subtlety in the definition of a spectral triple, however, this is not true. It does follow, however, that the triple $((\mathcal{K}^\infty)^+, H_{ev} \oplus H_{odd}, \mathcal{D})$, where $\mathcal{D}$ is defined in \cite{KMP1},  does indeed define a spectral triple over $\mathcal{K}^+$, and this triple pairs nontrivially with the class $[P_{0,0}]_0$ in $K_0(\mathcal{K}^+)$ as will be verified below. 

Let $H_w$ be the Hilbert space consisting of infinite series of operators 
\begin{equation}\label{f_expansion}
f = \sum_{n\geq 0}U^n f_n(\K)+\sum_{n< 0}f_n(\K)(U^*)^{-n} \:
\end{equation} 
satisfying:
\begin{equation*}
\|f\|_w^2 = \sum_{n \geq 0} \sum_{k \geq 0} w(k)|f_n(k)|^2 + \sum_{n < 0} \sum_{k \geq 0} w(k-n)|f_n(k)|^2 < \infty
\end{equation*}
where $w(k) > 0$ for all $k \in \Z_{\geq 0}$ and 
$$\sum_{k=0}^{\infty}{w(k)}=1.$$ 

Consider an operator 
$$D:H_w\supseteq \textrm{dom}(D) \to H_{w'}$$
on its maximal domain in $H_w$:
\begin{equation*}
\textrm{dom}(D) = \left\{ f \in H_w\subseteq\mathcal{F}: Df \in H_{w'} \right\}. 
\end{equation*}
given by the formula:
\begin{equation*}
Df = U\beta(\K)f - fU\alpha(\K),
\end{equation*} 
where $\beta(k)$ is a sequence such that:
\begin{equation*}
\lim_{k \to \infty}\beta(k+1) - \beta(k) := \beta_\infty\ne 0.
\end{equation*}
and where $\alpha(k)$ is a sequence such that:
\begin{equation} \label{one}
\sum_{k=0}^\infty|\beta(k) - \alpha(k)|^2w'(k) < \infty.
\end{equation}
We assume that for every $k$ we have:
\begin{equation} \label{three}
    \alpha(k), \beta(k)\neq 0.
\end{equation}
It is convenient to write $\alpha(k)$ as:
    \begin{equation} \label{four}
        \alpha(k) = \beta(k)\frac{\mu(k+1)}{\mu(k)} \textrm{   where   } \mu(0) = 1.
    \end{equation}
Additional assumptions are as follows:
\begin{itemize}
\item 
\begin{equation} \label{five}
\left|\frac{\beta(k) \cdots \beta(k+n)}{\beta(j) \cdots \beta(j+n)}\right| \leq \textrm{const} \textrm{ for all } k \leq j\textrm{, } n\in\Z_{\geq0},
\end{equation}
\item
\begin{equation} \label{six}
\sum_{k=0}^{\infty}\sum_{j=0}^{\infty}\frac{1}{(\textrm{max}\,(j,k) + 1)^2}\left|\frac{\mu(j)}{\mu(k)}\right|^2\frac{w(k)}{w'(j)} < \infty,
\end{equation}
\item
\begin{equation} \label{seven}
\textrm{there exists N such that }
\sum_{k=0}^{\infty} \frac{(1+k)^{2n}}{|\mu(k)|^{2}}w(k) <\infty\textrm{ for } n<N 
\textrm{ and infinite for } n\geq N. 
\end{equation}
\end{itemize}
Then $D$ on dom$(D)$ has a compact parametrix and its index is equal to $N$.

Let $\delta_\beta: (\mathcal{K}^\infty)^+\to (\mathcal{K}^\infty)^+$ be the continuous derivation given by:
\begin{equation*}
\delta_\beta(a):=[\beta(\K),a]\,.
\end{equation*}
Then, for every $a\in(\mathcal{K}^\infty)^+$, and for every $f\in(\mathcal{K}^\infty)^+$, considered as an element of both $H_w$ and $H_{w'}$, we have the formula
$$D\rho_w(a)f -\rho_{w'}(a)Df = \rho_{w'}(\delta_\beta(a))f\,,$$
where the representation $\rho_w :\mathcal{K}^+ \to B(H_w)$ of $\mathcal{K}^+$ in $H_w$ is given by left multiplication: 
$$\rho_w(a)f=af.$$
The issue here (which was overlooked in \cite{KMP1}) is that the right-hand side of the above equation, namely:
\begin{equation*}
H_w\ni f\mapsto \rho_{w'}(\delta_\beta(a))f\in H_{w'}\,,
\end{equation*}
might possibly be an unbounded operator. However, if we assume that $w(k)$ and $w'(k)$ decrease polynomially, then the above operator has RD matrix coefficients in the natural bases and so it is bounded, and even a smooth compact operator for all $a\in(\mathcal{K}^\infty)^+$. Examples of such parameters were discussed in \cite{KMP1}. 

Let $H = H_{w'} \bigoplus H_w$, with grading $\Gamma \big|_{H_{w'}} =1$ and $\Gamma \big|_{H_w} = -1$. Define a representation $\rho :\mathcal{K}^+ \to B(H)$ of $\mathcal{K}^+$ in $H$ by the formula:
$$\rho(a) = (\rho_{w'}(a),\rho_w(a)),$$
and let
\begin{equation*}
\dcal = \left[
\begin{array}{cc}
0 & D \\
D^* & 0
\end{array}\right],
\end{equation*}
so that $\rho(a)$ are even and $\dcal$ is odd with respect to grading $\Gamma$. 
With the above notation, $((\mathcal{K}^\infty)^+,H,\dcal)$ forms an even spectral triple over $\mathcal{K}^+$.

To compute the pairing
$$
\langle [P_{0,0}]_0 , ((\mathcal{K}^\infty)^+,H,\dcal) \rangle \,,
$$
we first describe the range of $\rho_w(P_{0,0})$.  For $f\in H_w$ we have:
\begin{equation*}
P_{0,0}f=P_{0,0}\left(\sum_{n\geq 0}U^n f_n(\K)+\sum_{n< 0}f_n(\K)(U^*)^{-n}\right)=\sum_{n\leq0}f_n(0)P_{0,0}(U^*)^{-n}.
\end{equation*}
We use it to compute the restriction of $D$:
\begin{equation*}
\rho_{w'}(P_{0,0})D\rho_{w}(P_{0,0})f=P_{0,0}U\beta(\K)P_{0,0}f - P_{0,0}fU\alpha(\K)\,.
\end{equation*}
The first term above vanishes since $P_{0,0}U=0$. The second can be computed as follows:
\begin{equation*}
- P_{0,0}fU\alpha(\K)=-\sum_{n\leq0}f_n(0)P_{0,0}(U^*)^{-n}U\alpha(\K)=-\sum_{n\leq0}\alpha(-n)f_{n-1}(0)P_{0,0}(U^*)^{-n}.
\end{equation*}
From this formula we see that the operator
\begin{equation*}
\rho_{w'}(P_{0,0})D\rho_{w}(P_{0,0}): \textrm{Ran}( \rho_{w}(P_{0,0})) \to \textrm{Ran}(\rho_{w'}(P_{0,0}))
\end{equation*}
is unitarily equivalent to the operator
\begin{equation*}
d:\ell^2_w\supseteq \textrm{dom}(d)\to\ell^2_{w'}\,,
\end{equation*}
where we used the following notation:
$$\ell^2_w=\left\{f(k): \sum_{k=0}^\infty|f(k)|^2w(k)<\infty\right\},
$$
given on its maximal domain dom$(d)$ by
\begin{equation*}
df(k)=-\alpha(k)f(k+1)\,.
\end{equation*}
Clearly $D$ is a weighted unilateral shift with index equal to 1. This gives
$$
\langle [P_{0,0}]_0 , ((\mathcal{K}^\infty)^+,H,\dcal) \rangle=1 \,.
$$
If we choose the parameters so that index of $D$ is zero, then the above spectral triple $((\mathcal{K}^\infty)^+,H,\dcal)$ gives a generator of $K^0(\mathcal{K})$.

\subsubsection{Fredholm modules over $\mathcal{T}$} Again, it follows from the table that there are no topologically nontrivial odd Fredholm modules over $\mathcal{T}$. However, an even Fredholm module over $\mathcal{T}$ whose class generates $K^0(\mathcal{T})$ can be seen as a pullback of the Fredholm module in Proposition \ref{ev FM circle} by the map $K^0(q): K^0(C(\R / \Z)) \to K^0(\mathcal{T})$ in $K$-homology induced by the quotient $q: \mathcal{T} \to C(\R / \Z) $. The following proposition formalizes this observation. 
\begin{prop}
Consider the Fredholm module $(\{0\}\oplus \C, \rho_0, F_0, \Gamma)$ defined in Proposition \ref{ev FM circle}. $(\{0\}\oplus \C, \rho_0 \circ q, F_0, \Gamma)$ determines an even Fredholm module over $\mathcal{T}$ that generates $K^0(\mathcal{T})$. 
\end{prop}
\begin{proof}
Since $q: \mathcal{T} \to C(\R / \Z)$ is a $*$-homomorphism, $\rho_0 \circ q$ is indeed a $*$-representation. The required properties of $F$ have already been checked in Proposition \ref{ev FM circle}. Note that $ \rho_0 \circ q(I) = \rho_0(1)$. Hence, the pairing between the generator $[I]_0$ of $K_0(\mathcal{T})$ with the class of the Fredholm module $(\{0\}\oplus \C, \rho_0 \circ q, F_0, \Gamma)$ over $\mathcal{T}$ is the same as the pairing between the generator $[1]_0$ of $K_0(C(\R / \Z))$ and the class of the Fredholm module $(\{0\}\oplus \C, \rho_0, F_0, \Gamma)$ over $C(\R / \Z)$. This pairing was computed to be $1$ in Proposition \ref{ev FM circle}. It follows by the same reasoning as before that the class of $(\{0\}\oplus \C, \rho_0 \circ q, F_0, \Gamma)$  generates $K^0(\mathcal{T})$. 
\end{proof}

\subsection{Cyclic Homology}
Of various cyclic theories the local cyclic theory of Puschnigg \cite{Pu} has perhaps the best properties.
For C$^*$-algebras the standard periodic cyclic theory gives only basically trivial and pathological results, while local cyclic theory is well-behaved for C$^*$-algebras. The other advantage of local cyclic theory is that it is stable when passing to smooth subalgebras.

The Universal Coefficient Theorem for local cyclic homology \cite{Me} asserts that Chern-Connes character is an isomorphism
\begin{equation*}
K_*(A)\otimes\C\to HL_*(A)
\end{equation*}
if $A$ belongs to the bootstrap category. Here $ HL_*(A)$ are the local cyclic homology groups of $A$. It follows that we have the following descriptions of local cyclic homology groups.
\begin{prop}
The local cyclic homology groups of $C(\R / \Z), \mathcal{K}$, and $\mathcal{T}$ are equal to the local cyclic homology groups of their corresponding smooth subalgebras and are summarized in the table below. 
\begin{center}
\begin{tabular}{ |c|c| } 
 \hline
 $HL_0(C(\R / \Z))$ & $\C$ \\ 
 \hline
 $HL_0(\mathcal{K})$ & $\C$ \\ 
 \hline
  $HL_0(\mathcal{T})$ & $\C$  \\ 
 \hline
 $HL_1(C(\R / \Z))$ & $\C$ \\ 
 \hline
 $HL_1(\mathcal{K})$ & $0$ \\ 
 \hline
  $HL_1(\mathcal{T})$ & $0$ \\ 
  \hline
\end{tabular}
\end{center}
\end{prop}
\begin{proof} As mentioned above, the statements are a simple consequence of the general properties of the $HL_*$ groups.
\end{proof}

\end{document}